\documentclass[12pt]{amsart}
\usepackage[english]{babel}

\usepackage{amsmath}
\usepackage{graphicx}
\usepackage[colorlinks=true, allcolors=blue]{hyperref}
\usepackage{xcolor}

\usepackage[numbers]{natbib}
\usepackage{enumerate}
\usepackage{amsthm,amssymb,amsfonts,mathrsfs,xfrac,faktor,esint}
\usepackage{tikz-cd}
\usepackage{dsfont}
\usepackage{multirow}
\usepackage{tabularx}
\usepackage{soul}

\numberwithin{equation}{section}

\newtheorem{theorem}{Theorem}[section]
\newtheorem{cor}[theorem]{Corollary}
\newtheorem{lemma}[theorem]{Lemma}
\newtheorem{prop}[theorem]{Proposition}

\theoremstyle{definition}
\newtheorem{definition}[theorem]{Definition}

\theoremstyle{remark}
\newtheorem{remark}[theorem]{Remark}

\usepackage{forloop}
\newcommand{\defcal}[1]{\expandafter\newcommand\csname 
	c#1\endcsname{{\mathcal{#1}}}}
\newcommand{\defbb}[1]{\expandafter\newcommand\csname 
	b#1\endcsname{{\mathds{#1}}}}
\newcommand{\defbf}[1]{\expandafter\newcommand\csname 
	bf#1\endcsname{{\mathbf{#1}}}}
\newcommand{\deffr}[1]{\expandafter\newcommand\csname 
	frak#1\endcsname{{\mathfrak{#1}}}}
\newcommand{\defov}[1]{\expandafter\newcommand\csname 
	ov#1\endcsname{{\overline{#1}}}}
\newcommand{\deftil}[1]{\expandafter\newcommand\csname 
	til#1\endcsname{{\widetilde{#1}}}}
\newcommand{\defhat}[1]{\expandafter\newcommand\csname 
	hat#1\endcsname{{\widehat{#1}}}}
\newcommand{\defscr}[1]{\expandafter\newcommand\csname 
	scr#1\endcsname{{\mathscr{#1}}}}
\newcommand{\deftt}[1]{\expandafter\newcommand\csname 
	tt#1\endcsname{{\mathtt{#1}}}}
\newcommand{\deful}[1]{\expandafter\newcommand\csname 
	ul#1\endcsname{{\underline{#1}}}}
\newcommand{\defas}[1]{\expandafter\newcommand\csname 
	'#1\endcsname{{\'{#1}}}}
\newcommand{\defde}[1]{\expandafter\newcommand\csname 
	`#1\endcsname{{\`{#1}}}}
\newcommand{\defpl}[1]{\expandafter\newcommand\csname 
	^#1\endcsname{{\^{#1}}}}
\newcommand{\defrm}[1]{\expandafter\newcommand\csname 
	rm#1\endcsname{{\mathsf{#1}}}}
 \newcommand{\defdot}[1]{\expandafter\newcommand\csname 
	Dot#1\endcsname{{\Dot{#1}}}}

\newcounter{calBbCounter}
\forLoop{1}{26}{calBbCounter}{
	\edef\letter{\Alph{calBbCounter}}
	\expandafter\defcal\letter
	\expandafter\defbb\letter
	\expandafter\defbf\letter
	\expandafter\deffr\letter
	\expandafter\defov\letter
	\expandafter\deftil\letter
	\expandafter\defhat\letter
	\expandafter\defscr\letter
	\expandafter\deftt\letter
	\expandafter\deful\letter
	\expandafter\defas\letter
	\expandafter\defde\letter
	\expandafter\defpl\letter
	\expandafter\defrm\letter
    \expandafter\defdot\letter
}
\forLoop{1}{26}{calBbCounter}{
	\edef\letter{\alph{calBbCounter}}
	\expandafter\defbf\letter
	\expandafter\deffr\letter
	\expandafter\defov\letter
	\expandafter\deftil\letter
	\expandafter\defhat\letter
	\expandafter\defscr\letter
	\expandafter\deftt\letter
	\expandafter\deful\letter
	\expandafter\defas\letter
	\expandafter\defde\letter
	\expandafter\defpl\letter
	\expandafter\defrm\letter
    \expandafter\defdot\letter
}

\DeclareMathOperator{\rhs}{RHS}
\DeclareMathOperator{\lhs}{LHS}

\DeclareMathOperator{\dist}{dist}

\DeclareMathOperator{\Div}{div}
\DeclareMathOperator{\supp}{supp}

\DeclareMathOperator{\loc}{\mathsf{loc}}

\DeclareMathOperator{\sing}{sing}

\newcommand{\Cc}{C^\infty_{\mathsf{c}}}
\newcommand{\esssup}[1]{\underset{#1}{\mathrm{esssup}}}
\DeclareMathOperator{\weakstar}{weak^\ast}

\DeclareMathOperator{\mat}{Mat}

\newcommand{\I}{\mathds{1}}

\DeclareMathOperator{\sk}{SK}
\DeclareMathOperator{\sio}{SIO}

\DeclareMathOperator{\RH}{RH}

\title[Well-posedness and maximal regularity on tent spaces]{On well-posedness and maximal regularity for parabolic Cauchy problems on weighted tent spaces}
\author{Pascal Auscher}
\address{Universit{\'e} Paris-Saclay, CNRS, Laboratoire de Math\'{e}matiques d'Orsay, 91405 Orsay, France}
\email{pascal.auscher@universite-paris-saclay.fr}

\author{Hedong Hou}
\address{Universit{\'e} Paris-Saclay, CNRS, Laboratoire de Math\'{e}matiques d'Orsay, 91405 Orsay, France}
\email{hedong.hou@universite-paris-saclay.fr}

\date{July 29, 2024}

\begin{document}

\begin{abstract}
   We prove well-posedness in weighted tent spaces of weak solutions to the Cauchy problem $\partial_t u - \Div A \nabla u = f, u(0)=0$, where the source $f$ also lies in (different) weighted tent spaces, provided the complex coefficient matrix $A$ is bounded, measurable, time-independent, and uniformly elliptic. To achieve this,  we extend the theory of singular integral operators on tent spaces via off-diagonal estimates introduced by \cite{auscher12maxreg} to obtain estimates on solutions $u$, and also $\nabla u$, $\partial_t u$, and $\Div A \nabla u$ in weighted tent spaces, showing at the same time maximal regularity. Uniqueness follows from a different strategy using interior representation for weak solutions and boundary behavior. 
\end{abstract}
\maketitle
\tableofcontents
\section{Introduction}
\label{sec:intro}

The main goal of this work is to study some parabolic equations in weighted tent spaces. We prove the following well-posedness and maximal regularity result.

\begin{theorem}
    \label{thm:main} 
    Let $A \in L^\infty(\bR^n;\mat_n(\bC))$ be uniformly elliptic.
    Let $\beta>-1/2$ and $p_L(\beta)<p \leq \infty$. For any $f \in T^{p}_\beta$, there is a unique global weak solution $u \in T^{p}_{\beta+1}$ to the equation
    \begin{equation}
        \label{e:ic}
        \partial_t u(t,x) - \Div_x(A(x) \nabla_x u)(t,x) = f(t,x), \quad (t,x) \in (0,\infty) \times \bR^n,
    \end{equation} 
     Moreover, it satisfies the estimate
    \[ \|u\|_{T^p_{\beta+1}} + \|\nabla u\|_{T^p_{\beta+{1}/{2}}} \lesssim \|f\|_{T^p_\beta}, \]
    and both of $\partial_t u$ and $\Div(A\nabla u)$ belong to $T^{p}_\beta$ with estimate
    \[ \|\partial_t u\|_{T^p_\beta} + \|\Div(A\nabla u)\|_{T^p_\beta} \lesssim \|f\|_{T^p_\beta}.  \]
\end{theorem}

The definition of weak solutions is recalled in Section \ref{sec:inhomo} and that of the weighted tent spaces $T^p_\beta$ in Section \ref{ssec:tent}. For the time being, it suffices to know that  $p$ is an integrability index and $s:=2\beta+1$ is a regularity index. The number $p_L(\beta)$ will be given later and depends on the elliptic operator $L:=-\Div(A\nabla)$. We mention that this number is sharp in the case of the heat equation, with value $\frac {n}{n+2\beta+2}$.

Let us first clarify the absence of an initial condition in the statement. Actually, when $\beta>-1/2$ and $0<p\le \infty$, we shall show that all $T^p_{\beta+1}$-functions  have zero Whitney trace. This notion of trace,  which differs from the usual one taking completion from the usual restriction of test functions,  is a sort of non-tangential convergence at $t=0$: it consists in taking the limit of averages on parabolic Whitney cubes as $t \to 0$ for a.e. $x \in \bR^n$, see Theorem \ref{thm:ic_sol} for a precise description.  In particular, the solution in the statement must have zero Whitney trace. It also implies that there is no non-trivial global weak solution in $T^{p}_{\beta+1}$ to $\partial_tu - \Div(A\nabla u)=0$, and hence the initial value problem with non-zero datum cannot be posed in $T^p_{\beta+1}$ when $\beta>-1/2$. \\

That the initial condition must be zero suggests that one may solve the equation only using appropriate estimates for the Duhamel operator
\begin{equation}
    \cL_{1}(f)(t) := \int_0^t e^{-(t-s)L} f(s) ds,
    \label{e:formal_sol}
\end{equation}
its gradient $\nabla \cL_{1}$, and the maximal regularity operator
\begin{equation}
    \cM_L(f)(t) := \int_0^t Le^{-(t-s)L} f(s) ds.
    \label{e:def_M_L}
\end{equation}
Here, the semigroup is defined on $L^2(\bR^n)$ when realizing $L$ as a maximal accretive operator by Kato's theory. Note that the operator norms of the first two integrands in $L^2(\bR^n)$ have integrable singularities, while that of the maximal regularity operator is on the order of $|t-s|^{-1}$. Still, it is bounded on $L^2((0,\infty); L^2(\bR^n))$ by de Simon's theorem \cite{deSimon1964MRO_L2}. This corresponds here to the tent space $T^2_{0}$. Classical maximal regularity theory (\textit{i.e.}, boundedness of $\cM_L$ on $L^q((0,\infty); X)$ for $1<q<\infty$) also applies, provided the semigroup has the $R$-boundedness property and that $X$ is a UMD  space, see \cite{Weis2001LpMaxReg}. It produces what is usually called unique mild solutions to the abstract Cauchy problem
\begin{equation}
    \begin{cases}
        \partial_t u(t) + Lu(t) = f(t)  \\
        u(0)=0
    \end{cases}.
    \label{e:abstract-Cauchy}
\end{equation}
In the case of $L=-\Div(A\nabla)$,  $R$-boundedness on $X=L^p(\bR^n)$ only holds in a range of $p$ which can be limited \cite{Blunck-Kunstmann2003_CZNonIntegral}. Indeed, it is the interval $(p_-(L),p_{+}(L))$  introduced by \cite{Auscher2007Memoire} as the largest open interval  in $(1,\infty)$ so that the semigroup $(e^{-tL})_{t \ge 0}$, initially defined on $L^2(\bR^n)$, is uniformly bounded on $L^p(\bR^n)$. For negative Laplacian $-\Delta$, $p_-(-\Delta)=1$, $p_{+}(-\Delta)=\infty$. Examples show that $\frac{2n}{n\mp2}\mp \varepsilon$ (when $n\ge 3$) are best possible in full generality. 

The reader may wonder why we attempt to use tent spaces. The interest is that their norms first use local $L^2$ integrability in the upper half-space no matter what $p$ and $\beta$ are. Thus, we can be dispensed with checking $R$-boundedness and work with general non-smooth coefficients. To our knowledge, the first use of tent spaces (in the form of Carleson measures) to solve evolution equations is due to \cite{Koch-Tataru2001NS} on Navier--Stokes equations. See also a different proof in \cite{Auscher-Frey2017NS} and more recently \cite{Danchin-Vasilyev2023_Inhomogeneous-NS-tent} for inhomogeneous incompressible Navier--Stokes equations.

Getting away from the Laplacian, a series of pioneering works of the first author \textit{et al.} \cite{Auscher-Monniaux-Portal2012_MR,auscher12maxreg} proves boundedness of maximal regularity operators associated with such $L$ on a family of weighted tent spaces $T^{p}_\beta$, completely avoiding the notion of $R$-boundedness. Here we prove the following result. 
\begin{prop}
    \label{prop:ML_ext_tent}
    Let $\beta>-1/2$ and $p_L(\beta) < p \leq \infty$. Then, $\cM_L$ can be extended to a bounded operator on $T^{p}_\beta$. In addition, $\cL_{1}$ is bounded from $T^{p}_\beta$ to $T^{p}_{\beta+1}$ and $\nabla \cL_{1}$ is bounded from $T^{p}_\beta$ to $T^{p}_{\beta+1/2}$.
\end{prop}
We now define 
\begin{equation}
    p_L(\beta):=\frac{np_-(L)}{n+(2\beta+1)p_-(L)},
    \label{e:def_pL(beta)_intro}
\end{equation}
where $p_-(L)$ is introduced above. Observe that $\DotH^{2\beta+1,p_L(\beta)}(\bR^n)$ embeds in $L^{p_-(L)}(\bR^n)$ by Sobolev embedding, so $p_L(\beta)$ is in fact a Sobolev exponent of $p_-(L)$ at regularity $2\beta+1$. We also see that information of boundedness of the semigroup on $L^p(\bR^n)$ for $p<2$ helps improving the range. 

Compared to \cite{auscher12maxreg}, we improve here the lower limit of tent space boundedness  to $p_L(\beta)$ in all cases, thanks to a new extrapolation argument. We additionally treat here the operators $\cL_{1}$ and $\nabla \cL_{1}$. Some results in this direction with application towards stochastic analysis are in  \cite{Auscher-vanNeerven-Portal2014_StochasticMR} and \cite{Portal-Veraar2019_MR}. For this, we not only extend the singular integral theory to different (singular) kernels, but also correct an inaccuracy in the definition of singular integral operators on tent spaces. As we include different singularities,  it is somehow difficult to explain the correction without going again through the technical lemmas, so we take this opportunity to do it here. In the end, it streamlines the presentation and allows for improvement, see Section \ref{ssec:sio}.  

Correctly reframed using the point of view of weak solutions, this proposition gives us existence. 

Uniqueness relies on a method already used in \cite{auscher-monniaux-portal2019existence}, which can be thought as an elaboration of Green's identity in the context of rough coefficients. It allows one to show that any global (or in a strip) weak solution to \eqref{e:abstract-Cauchy} with null source term in a tent space satisfies an interior representation, which we may call  ``\textit{homotopy identity}",
\[ u(t)=e^{-(t-s)L} u(s) \quad \text{ in } \scrD'(\bR^n) \]  
when $0<s<t$, in the sense that
\[ \int_{\bR^n} u(t,x) \ovh(x) dx = \int_{\bR^n} u(s,x) \overline{((e^{-(t-s)L})^\ast h)}(x) dx \]
for any $h \in C_\rmc^\infty(\bR^n)$.  With this identity, the problem is reduced to understand the boundary behavior of solutions when $s\to 0$ and use regularity $(e^{-(t-s)L})^\ast h$ to $(e^{-tL})^\ast h$ in appropriate topology as $s \to 0$, which only depends on functional calculus of the operator $L$. \\

To finish this introduction, we mention the results addressed in a second work \cite{Auscher-Hou2024_HCL}. As we observed after Theorem \ref{thm:main}, the weak solutions to \eqref{e:ic} in $T^p_{\beta+1}$ have null Whitney trace when $2\beta+1>0$ (and also zero distributional limit). However, if $2\beta+1=0$, then the initial value problem can be uniquely posed with non-zero Whitney trace, even for time-dependent coefficients, see \cite{auscher-monniaux-portal2019existence} and \cite{Zaton2020}. The solution class is given by $\nabla u \in T^p_{0}$. For time-independent coefficients, we will generalize it to $-1 < 2\beta+1 < 1$, where the initial data can be taken with Sobolev regularity $2\beta+1$.

At the same time, other source terms in divergence form $\Div F$ can be considered, in the spirit of the Lions' theory.  The function $F$ will be taken in $T^p_{\beta+{1}/{2}}$ for $2\beta+1>-1$, extending  \cite{Auscher-Portal2023_Lions}  which considers $2\beta+1=0$ (but for non-autonomous equations). In particular, this exhibits a range of $(\beta,p)$ for which the Cauchy problem
\[ \begin{cases}
    \partial_t u - \Div(A \nabla u) = \Div F \\
    u(0)=u_0
\end{cases} \]
can be uniquely posed for solutions $\nabla u \in T^p_{\beta+1/2}$ provided $F \in T^p_{\beta+{1}/{2}}$ and $u_0 \in \DotH^{2\beta+1,p}$ 
when $-1<2\beta+1<1$. When $2\beta+1 <0$, the solution $u$ will also lie in $T^p_{\beta+1}$, but no maximal regularity estimate.

\subsection*{Organization}
\label{ssec:organization}
The paper is organized as follows. 

Section \ref{sec:sio} is devoted to the definitions of two classes of singular integral operators. We show that they are \textit{de facto} linked by duality. 

Section \ref{sec:sio_tent} is concerned with extension of these two classes of singular integral operators to tent spaces. Main results are given in Section \ref{ssec:results_sio_tent}. 

Sections \ref{sec:inhomo} and \ref{sec:mr} are main sections of this paper. Section \ref{ssec:results_ic} outlines the proof of the main theorem, Theorem \ref{thm:main}, about global well-posedness of \eqref{e:ic}, as a consequence of Theorems \ref{thm:ic_sol} and \ref{thm:unique_ic}. Theorem \ref{thm:ic_sol} utilises the theory of singular integral operators developed in Section \ref{sec:sio_tent} to establish existence through a series of weighted tent space estimates on the Duhamel solution $u$, its derivatives $\nabla u$, $\partial_t u$, and $\Div(A\nabla u)$. The proof is provided in Section \ref{ssec:exist_reg_ic} and Section \ref{sec:mr}, where maximal regularity estimates are established, and boundary behavior of $u$ is described. Theorem \ref{thm:unique_ic} states uniqueness of global weak solutions. Its proof is given in Section \ref{ssec:uniquess_ic}. 

\subsection*{Notation}
\label{ssec:notation}
Throughout the paper, we write 
\[ [q,r]:=\frac{1}{q}-\frac{1}{r} \]
for any $q,r \in (0,\infty]$, if there is no confusion with closed intervals. We write $X \lesssim Y$ (or $X \lesssim_A Y$, resp.) if $X \le CY$ with an irrelevant constant $C$ (or depending on $A$, resp.), and say $X \eqsim Y$ if $X \lesssim Y$ and $Y \lesssim X$. Write $\bR^{1+n}_+:=\bR_+ \times \bR^n = (0,\infty) \times \bR^n$. For any (Euclidean) ball $B \subset \bR^n$, write $r(B)$ for the radius of $B$.

Let $(X,\mu)$ be a measure space. For any measurable subset $E \subset X$ with finite positive measure and $f \in L^1(E,\mu)$, we write
\[ \fint_E f d\mu := \frac{1}{\mu(E)} \int_E f d\mu. \]
Write $\|\cdot\|_p$ as an abbreviation for the norm $\|\cdot\|_{L^p(X,\mu)}$.

We write $L^p(\bR^n)=L^p(\bR^n,dx)$ where $dx$ is the Lebesgue measure, and $L^2(\bR^n)$ is equipped with the usual (complex) inner product. For any $\beta \in \bR$, denote by $L^2_\beta(\bR^{1+n}_+)$ the space $L^2(\bR_+,t^{-2\beta} dt;L^2(\bR^n))$.

We use the sans-serif font $\rmc$ in the scripts of function spaces in short of ``with compact support" in the prescribed set, and $\loc$ if the prescribed property holds on all compact subsets of the prescribed set.

\subsection*{Acknowledgements}
The authors were supported by the ANR project RAGE ANR-18-CE40-0012. We would like to thank Luca Haardt and Jonas Lenz for pointing out a gap in one of our arguments and Luca Haardt also for a very careful reading to correct many typos. The second author would also like to thank Jing'er Ji for her long-lasting encouragements and discussions.
\section{Singular integral operators}
\label{sec:sio}

In this section, we rectify and generalize the abstract model for singular integral operators in \cite{auscher12maxreg} via off-diagonal decay. 

Let $\pi_1,\pi_2$ be the projections of $\bR^{1+n}=\bR\times\bR^n$ to $\bR$ and $\bR^n$ components, respectively. Denote by $\Delta$ the set $\{(t,s) \in (0,\infty) \times (0,\infty):t=s\}$ and write $\Delta^c:=((0,\infty) \times (0,\infty)) \setminus \Delta$. 

\begin{definition}[Off-diagonal decay]
\label{def:off-diagonal_decay}
Let $\kappa \in \bR, m \in \bN, 1 \leq q \leq r \leq \infty$, and $M>0$ be constants. Let $\{K(t,s)\}_{(t,s) \in \Delta^c}$ be a family of bounded operators on $L^2(\bR^n)$. We say that it has \textit{$L^q-L^r$ off-diagonal decay of type $(\kappa,m,M)$} if there exists a constant $C>0$ such that for any Borel sets $E,F \subset \bR^n$, $f \in L^2(\bR^n) \cap L^q(\bR^n)$, and a.e. $(t,s) \in \Delta^c$,
\begin{equation}
    \|\I_E K(t,s) \I_F f\|_r \leq C|t-s|^{-1+\kappa-\frac{n}{m}[q,r]} \left ( 1+\frac{\dist(E,F)^m}{|t-s|} \right )^{-M} \|\I_F f\|_q.
    \label{e:def_off-diagonal_decay}
\end{equation}
\end{definition}

\subsection{Singular kernels}
\label{ssec:sk}

Let $\scrL(X)$ be the space of bounded linear operators on a (quasi-)Banach space $X$.

\begin{definition}[Singular kernel]
\label{def:sk(kappa,m,q,M)}

Let $\kappa \in \bR, m \in \bN, 1 \leq q \leq \infty$, and $M>0$ be constants. We say that an operator-valued function $K:\Delta^c \rightarrow \scrL(L^2(\bR^n))$ is a \textit{singular kernel} (SK) \textit{of type $(\kappa,m,q,M)$}, if
\begin{enumerate}
    \item $K$ is strongly measurable;
    \item There exists some constant $C>0$ such that for a.e. $(t,s) \in \Delta^c$,
    \begin{equation}
        \|K(t,s)\|_{\scrL(L^2(\bR^n))} \leq C|t-s|^{-1+\kappa};
        \label{e:def_sk_bdd_K(t,s)_L2}
    \end{equation}
    \item The family $\{K(t,s)\}_{(t,s) \in \Delta^c}$ has off-diagonal $L^q-L^2$ decay of type $(\kappa,m,M)$ if $1 \leq q \leq 2$, and off-diagonal $L^2-L^q$ decay of type $(\kappa,m,M)$ if $2 \leq q \leq \infty$.
\end{enumerate}
Let $\sk^\kappa_{m,q,M}$ be the set of singular kernels of type $(\kappa,m,q,M)$ and $\sk^\kappa_{m,q,\infty}$ be the intersection $\bigcap_{M>0} \sk^\kappa_{m,q,M}$.
\end{definition}

\begin{remark}
\label{rem:sk_truncate}
    The class $\sk^\kappa_{m,q,M}$ is stable with respect to truncation in the sense that, if $K \in \sk^\kappa_{m,q,M}$, then $\I_A K \in \sk^\kappa_{m,q,M}$ for any measurable set $A \subset \Delta^c$. The two cases, $A=\{t<s\}$ and $A=\{t>s\}$ are of particular interest.
\end{remark}

Denote by $T^\ast$ the adjoint of a linear operator $T$ on $L^2(\bR^n)$.

\begin{lemma}
\label{lemma:sk_duality_off-diagonal}
Let $1 \leq q \leq r \leq \infty$. Let $\{K(t,s)\}_{(t,s) \in \Delta^c}$ be a family of bounded operators on $L^2(\bR^n)$ with $L^q-L^r$ off-diagonal decay of type $(\kappa,m,M)$. Then $\{K(s,t)^\ast\}_{(t,s) \in \Delta^c}$ has $L^{r'}-L^{q'}$ off-diagonal decay of type $(\kappa,m,M)$.
\end{lemma}

\begin{proof}
Let $E,F \subset \bR^n$ be two Borel sets. For any $f \in L^2(\bR^n) \cap L^{r'}(\bR^n)$ and a.e. $(t,s) \in \Delta^c$,
\begin{align*}
\| \I_F K(s,t)^\ast (\I_E f)\|_{q'}
 &= \sup_\phi \left \langle \I_F \phi, K(s,t)^\ast (\I_E f) \right \rangle_{L^2(\bR^n)} \\ 
 &= \sup_\phi \left \langle K(s,t)(\I_F \phi), \I_E f \right \rangle_{L^2(\bR^n)} \\ 
 &\leq \sup_\phi \|\I_E K(s,t) (\I_F \phi)\|_r \|\I_E f\|_{r'} \\
 &\lesssim |t-s|^{-1+\kappa-\frac{n}{m}[q,r]} \left ( 1+\frac{\dist(E,F)^m}{|t-s|} \right )^{-M} \|\I_E f\|_{r'},
\end{align*}
where $\phi$ is taken in $L^2(\bR^n) \cap L^q(\bR^n)$ with support in $F$ and $\|\phi\|_q=1$. Note that $[q,r]=[r',q']$ to conclude the proof.
\end{proof}

\begin{cor}
\label{cor:sk_duality}
If $K$ belongs to $\sk^\kappa_{m,q,M}$, then $K^\ast:(t,s) \mapsto K(s,t)^\ast$ belongs to  $\sk^\kappa_{m,q',M}$.
\end{cor}

\begin{proof}
    The strong measurability of $K^\ast$ follows from Pettis' measurability theorem (see \cite[Theorem 1.1.6]{Hytonen-NVW2016BanachSpaces}), since $L^2(\bR^n)$ is separable. The boundedness condition is obvious, while the off-diagonal decay follows from Lemma \ref{lemma:sk_duality_off-diagonal}. 
\end{proof}

\subsection{Integrals associated with singular kernels}
\label{ssec:si_of_sk}

For any function $f$ on $\bR^{1+n}_+$, denote by $f(s)$ the function $x \mapsto f(s,x)$. Let us first consider a non-singular integral. Let $K$ be in $\sk^\kappa_{m,q,M}$ with $\kappa>0$. For any $\beta>-1/2$ and $f \in L^2_\beta(\bR^{1+n}_+)$, we claim that the integral 
\begin{equation}
    \int_0^t K(t,s)f(s) ds
    \label{e:def_si+}
\end{equation}
makes sense as a Bochner integral on $L^2(\bR^n)$ for a.e.~$t \in (0,\infty)$. Indeed, \eqref{e:def_sk_bdd_K(t,s)_L2} implies
\begin{align*}
    &\left (\int_0^\infty \left ( t^{-\beta-\kappa} \int_0^t \|K(t,s)f(s)\|_2 ~ds \right )^2 dt \right )^{1/2} \\ 
    &\quad \lesssim \left (\int_0^\infty \left ( \int_0^t s^{\beta+\frac{1}{2}} t^{-\beta-\kappa+\frac{1}{2}} (t-s)^{-1+\kappa} \|s^{-\beta+\frac{1}{2}} f(s)\|_2 \frac{ds}{s} \right )^2 \frac{dt}{t} \right )^{1/2}.
\end{align*}
Write $k(t,s)=\I_{\{s<t\}}(t,s) s^{\beta+\frac{1}{2}} t^{-\beta-\kappa+\frac{1}{2}} (t-s)^{-1+\kappa}$. For a.e. $t>0$,
\[ \int_0^\infty k(t,s) \frac{ds}{s} = \int_0^1 \lambda^{\beta+\frac{1}{2}} (1-\lambda)^{\kappa-1} \frac{d\lambda}{\lambda} \leq C(\beta,\kappa). \] With the same conditions, for a.e. $s>0$,
\[ \int_0^\infty k(t,s) \frac{dt}{t} = \int_0^\infty \lambda^{\kappa-1} (\lambda+1)^{-\beta-\kappa-\frac{1}{2}} d\lambda \leq C(\beta,\kappa). \]
Schur test hence implies
\[ \left (\int_0^\infty \left ( t^{-\beta-\kappa} \int_0^t \|K(t,s)f(s)\|_2 ~ds \right )^2 dt \right )^{1/2} \lesssim \|f\|_{L^2_\beta(\bR^{1+n}_+)} < \infty. \]
This proves our claim. It also shows that the function 
$$(t,y) \mapsto \left(\int_0^t K(t,s)f(s)ds\right)(y)$$
belongs to $L^2_{\beta+\kappa}(\bR^{1+n}_+)$. Thanks to the strong measurability of $K$, we have the pointwise evaluation
(see \cite[Proposition 1.2.25]{Hytonen-NVW2016BanachSpaces}) in the sense that for a.e.~$(t,y) \in \bR^{1+n}_+$,
\begin{equation}
    \left ( \int_0^t K(t,s)f(s) ds \right )(y) = \int_0^t (K(t,s)f(s))(y) ds.
    \label{e:def_si+_pointwise_evaluation}
\end{equation}

For $f \in L^2_{-\beta-\kappa}(\bR^{1+n}_+)$, by a duality argument, we can also define the $L^2(\bR^n)$-valued Bochner integral 
\begin{equation}
    \int_t^\infty K(t,s)f(s) ds
    \label{e:def_si-}
\end{equation}
for a.e.~$t \in (0,\infty)$, and obtain similarly
\begin{equation}
    \left( \int_t^\infty K(t,s)f(s) ds \right) (y) = \int_t^\infty (K(t,s)f(s))(y) ds
    \label{e:si-_pointwise}
\end{equation}
for a.e.~$(t,y) \in \bR^{1+n}_+$. We summarize the above discussion in the following lemma.

\begin{lemma}
    \label{lem:ext_L2_kappa>0}
    Let $K$ be in $\sk^\kappa_{m,q,M}$ with $\kappa>0$. For any $\beta>-1/2$,
    \begin{enumerate}
        \item \eqref{e:def_si+} defines a bounded operator from $L^2_\beta(\bR^{1+n}_+)$ to $L^2_{\beta+\kappa}(\bR^{1+n}_+)$, and \eqref{e:def_si+_pointwise_evaluation} holds;
        \item \eqref{e:def_si-} defines a bounded operator from $L^2_{-\beta-\kappa}(\bR^{1+n}_+)$ to $L^2_{-\beta}(\bR^{1+n}_+)$, and \eqref{e:si-_pointwise} holds.
    \end{enumerate}
\end{lemma}

We then turn to $\kappa \leq 0$, where a non-integrable singularity occurs at $t=s$. Still, as we are interested in handling functions of $(t,x)$, we  adopt a refined  strategy based on the following lemma. A subset $Q \subset \bR^{1+n}_+$ is called a \textit{rectangle} if it is of product form $Q=\pi_1(Q) \times \pi_2(Q)$. It is called a \textit{bounded} rectangle if both components are bounded.

\begin{lemma}
\label{lemma:sk_norm_schur}
Let $K\in \sk^\kappa_{m,q,M}$ with $\kappa \leq 0$ and $\frac{n}{m}|[q,2]|-\kappa < M \leq \infty$. Let $Q_0,Q_1$ be two bounded rectangles in $\bR^{1+n}_+$ with
\[ \max\{\dist(\pi_i(Q_0),\pi_i(Q_1)):i=1,2\} \geq \epsilon > 0. \]
For a.e.~$(t,s) \in \bR_+ \times \bR_+$, define
\[ \omega_{Q_0,Q_1}(t,s):=\| \I_{Q_0}(t) K(t,s) \I_{Q_1}(s) \|_{\scrL(L^2(\bR^n))}. \]
Then, it holds that
\begin{align*}
    \int_0^\infty \omega_{Q_0,Q_1}(t,s) ds &\leq C(\epsilon,Q_0,Q_1,\kappa,m,M,q), \\ 
    \int_0^\infty \omega_{Q_0,Q_1}(t,s) dt &\leq C(\epsilon,Q_0,Q_1,\kappa,m,M,q).
\end{align*}
\end{lemma}

\begin{proof}
By symmetry, it suffices to prove the first inequality. Recall that $\I_{Q_0}(t)(x)= \I_{\pi_1 (Q_0)}(t) \I_{\pi_2(Q_0)}(x)$ and  
$\I_{Q_1}(s)(y)= \I_{\pi_1 (Q_1)}(s) \I_{\pi_2(Q_1)}(y)$.

If $\dist(\pi_1(Q_0),\pi_1(Q_1)) \geq \epsilon/2$, for any $t \in \pi_1(Q_0)$, we have
\[ \int_0^\infty \omega_{Q_0,Q_1}(t,s) ds \lesssim \int_{\pi_1(Q_1)} |t-s|^{-1+\kappa} ds \leq (\epsilon/2)^{\kappa-1} |\pi_1(Q_1)|. \]
Otherwise, $\dist(\pi_1(Q_0),\pi_1(Q_1)) \leq \epsilon/2$ but then $\dist(\pi_2(Q_0),\pi_2(Q_1)) \geq \epsilon$. Off-diagonal decay of $K(t,s)$ implies
\[ \omega_{Q_0,Q_1}(t,s) \lesssim
\begin{cases}
\epsilon^{-mM} |t-s|^{-1+\kappa-\frac{n}{m}[q,2]+M} |\pi_2(Q_1)|^{[q,2]} & \text{ if } 1 \leq q \leq 2 \\ 
\epsilon^{-mM} |t-s|^{-1+\kappa-\frac{n}{m}[2,q]+M} |\pi_2(Q_0)|^{[2,q]} & \text{ if } 2 \leq q \leq \infty
\end{cases} \]
by H\"older's inequality. In summary, we get
\[ \int_0^\infty \omega_{Q_0,Q_1}(t,s) ds \lesssim_{Q_0,Q_1,q} \epsilon^{-mM} \int_{\pi_1(Q_1)} |t-s|^{-1+\kappa-\frac{n}{m}|[q,2]|+M } ds, \]
which converges as $M>\frac{n}{m}|[q,2]|-\kappa$.
\end{proof}

Let $L^2_\rmb(\bR^{1+n}_+)$ be the subspace of $L^2(\bR^{1+n}_+)$ consisting of functions with bounded support in $\overline{\bR^{1+n}_+}$. For any $f \in L^2_\rmb(\bR^{1+n}_+)$, define $\pi(f):=\pi_1(\supp f) \times \pi_2(\supp f)$ as a bounded rectangle. Let $K,\kappa$, and $M$ be as in Lemma \ref{lemma:sk_norm_schur}. For any bounded rectangle $Q \subset \bR^{1+n}_+$ with $\dist(Q,\pi(f))>0$ and a.e.~$t \in \pi_1(Q)$,
\[ \int_0^\infty \I_{\pi_2(Q)} K(t,s) f(s) ds \]
is defined as a Bochner integral valued in $L^2(\bR^n)$, since
\[ \int_0^\infty \| \I_{\pi_2(Q)} (K(t,s)f(s))\|_2 ~ds \leq \int_0^\infty \omega_{Q,\pi(f)}(t,s) \|f(s)\|_2 ~ds, \]
which is finite thanks to Lemma \ref{lemma:sk_norm_schur} and Schur test. Pointwise evaluation also holds, \textit{i.e.}, for a.e.~$(t,y) \in Q$,
\[ \left( \int_0^\infty \I_{\pi_2(Q)} K(t,s) f(s) ds \right) (y) = \int_0^\infty \I_{\pi_2(Q)}(y) (K(t,s) f(s))(y) ds. \]
Furthermore, if $Q,Q' \subset (\bR^{1+n}_+ \setminus \pi(f))$ are bounded rectangles with $Q \cap Q' \ne \varnothing$, then for a.e. $(t,y) \in Q \cap Q'$, we have
\begin{equation}
    \int_0^\infty \I_{\pi_2(Q)}(y) (K(t,s)f(s))(y) dy = \int_0^\infty \I_{\pi_2(Q')}(y) (K(t,s) f(s))(y) ds.
    \label{e:si_Q_cap_Q'-identification}
\end{equation}
Indeed, as $t\in \pi_1(Q) \cap \pi_1(Q')$, we have when $x\in \pi_2(Q) \cap \pi_2(Q')$
\[ \lhs = \int_0^\infty \I_{\pi_2(Q')}(y) \I_{\pi_2(Q)}(y) (K(t,s) f(s))(y) ds = \rhs. \]
Thus, we can define the function
\[ (t,y) \mapsto \int_0^\infty (K(t,s)f(s))(y) ds \]
almost everywhere on $\pi(f)^c$ as follows. Pick an \textit{exhaustion} of $\pi(f)^c$ by bounded rectangles $\{Q_i\}_{i \in \bN}$, \textit{i.e.}, $\bigcup_i Q_i = \pi(f)^c$. Then, we set almost everywhere
\begin{equation}
    \int_0^\infty (K(t,s)f(s))(y) ds := \int_0^\infty \I_{\pi_2(Q_i)}(y) (K(t,s) f(s))(y) ds,
    \label{e:def_SI_of_SK}
\end{equation}
if $(t,y) \in Q_i$. This definition makes sense almost everywhere and is clearly independent of the choice of exhaustion $\{Q_i\}_{i \in \bN}$ by \eqref{e:si_Q_cap_Q'-identification}. We summarize the above discussion in the following lemma.
\begin{lemma}
    \label{lem:sk_of_si_kappa<0}
    Let $\kappa \le 0$, $\frac{n}{m}|[q,2]|-\kappa < M \le \infty$, and $K \in \sk^\kappa_{m,q,M}$. Then for any $f \in L^2_\rmb(\bR^{1+n}_+)$,
    \[ (t,x) \mapsto \int_0^\infty (K(t,s) f(s))(x) ds \]
    defines a function in $L^2_{\loc}(\overline{\bR^{1+n}_+} \setminus \pi(f))$. Moreover, for any $g \in L^2_\rmb(\bR^{1+n}_+)$ with $\pi(g) \cap \pi(f) = \varnothing$,
    \[ (t,s) \mapsto \langle K(t,s)f(s), g(t) \rangle_{L^2(\bR^n)} \]
    defines a function in $L^1((0,\infty) \times (0,\infty))$, and
    \begin{align*}
        &\int_0^\infty \int_0^\infty \langle K(t,s)f(s), g(t) \rangle_{L^2(\bR^n)} dsdt \\
        &\quad = \int_0^\infty \int_{\bR^n} \left( \int_0^\infty (K(t,s)f(s))(x) ds \right) \cdot \ovg(t,x) dtdx.
    \end{align*}
\end{lemma}

\begin{proof}
    The first part follows from \eqref{e:def_SI_of_SK}, and the second part follows from  the first part by writing
    \[ \langle K(t,s)f(s), g(t) \rangle_{L^2(\bR^n)} = \langle \I_{\pi(g)}(t) K(t,s) \I_{\pi(f)}(s)f(s), g(t) \rangle_{L^2(\bR^n)} \]
    and using pointwise evaluation and Fubini's theorem.
\end{proof}

\subsection{Singular integral operators}
\label{ssec:sio}
Define
\begin{equation}
    M_{\kappa,q}:=\begin{cases}
    0 & \text{ if } \kappa>0 \\ 
    \frac{n}{m}|[q,2]|-\kappa & \text{ if } \kappa \leq 0 
    \end{cases}\ .
    \label{e:def_M_kappa}
\end{equation}

\begin{definition}[SIO of type $+(\kappa,m,q,M)$]
\label{def:sio+}
Let $\kappa \in \bR, m \in \bN, 1 \leq q \leq \infty$, and $M_{\kappa,q} < M \leq \infty$ be constants. An operator $T$ is called a \textit{singular integral operator} (SIO) \textit{of type $+(\kappa,m,q,M)$}, if
\begin{enumerate}
    \item $T$ is a bounded operator from $L^2(\bR^{1+n}_+)$ to $L^2_\kappa(\bR^{1+n}_+)$;
    \item There exists $K \in \sk^\kappa_{m,q,M}$, called the \textit{kernel} of $T$, such that the representation
    \begin{equation}
        Tf(t,y) = \int_0^t (K(t,s) f(s))(y) ds
        \label{e:def_T-representation}
    \end{equation}
    holds
    \begin{itemize}
        \item for any $f \in L^2(\bR^{1+n}_+)$ and a.e. $(t,y) \in \bR^{1+n}_+$, if $\kappa>0$;
        \item for any $f \in L^2_\rmb(\bR^{1+n}_+)$ and a.e. $(t,y) \in \pi(f)^c$, if $\kappa \leq 0$.
    \end{itemize}
\end{enumerate}
Denote by $\sio^{\kappa+}_{m,q,M}$ the set of all SIOs of type $+(\kappa,m,q,M)$.
\end{definition}

There is a slight abuse of terminology since the operator is not singular when $\kappa>0$.

The next class of singular integral operators is concerned with negative type. We shall present results for it since it is instrumental for duality. It can be used to obtain estimates for backward Cauchy problems, but we leave to the reader the care of stating them.

\begin{definition}[SIO of type $-(\kappa,m,q,M)$]
\label{def:sio-}
Let $\kappa \in \bR, m \in \bN, 1 \leq q \leq \infty$, and $M_{\kappa,q} < M \leq \infty$ be constants. An operator $T$ is called a \textit{singular integral operator of type $-(\kappa,m,q,M)$}, if
\begin{enumerate}
    \item $T$ is a bounded operator from $L^2_{-\kappa}(\bR^{1+n}_+)$ to $L^2(\bR^{1+n}_+)$;
    \item There exists $K \in \sk^\kappa_{m,q,M}$, called the \textit{kernel} of $T$, such that the representation
    \begin{equation}
        Tf(t,y) = \int_t^{\infty} (K(t,s)f(s))(y) ds
        \label{e:def_sio-_T-representation}
    \end{equation}
    holds
    \begin{itemize}
        \item for any $f \in L^2_{-\kappa}(\bR^{1+n}_+)$ and a.e. $(t,y) \in \bR^{1+n}_+$, if $\kappa>0$;
        \item for any $f \in L^2_{-\kappa}(\bR^{1+n}_+)$ with bounded support in $\overline{\bR^{1+n}_+}$ and a.e. $(t,y) \in \pi(f)^c$, if $\kappa \leq 0$.
    \end{itemize}
\end{enumerate}
Denote by $\sio^{\kappa-}_{m,q,M}$ the set of all SIOs of type $-(\kappa,m,q,M)$.
\end{definition}

For $\kappa>0$, the integrals in \eqref{e:def_T-representation} and \eqref{e:def_sio-_T-representation} are well-defined by Lemma \ref{lem:ext_L2_kappa>0}. For $\kappa \leq 0$, one uses Lemma \ref{lem:sk_of_si_kappa<0}. In the case $\kappa=0$, compared with Definition 2.1 in \cite{auscher12maxreg}, our definitions seem more restrictive, but in fact, \eqref{e:def_T-representation} and \eqref{e:def_sio-_T-representation} are all what is needed in their proofs to obtain tent space estimates.

These two types of SIOs are linked via $L^2(\bR^{1+n}_+)$-duality given by
\[ \langle f,g \rangle_{L^2(\bR^{1+n}_+)} := \int_{\bR^{1+n}_+} f(t,y) \ovg(t,y) dtdy. \]

\begin{prop}
\label{prop:sio_duality}
    Let $\kappa \in \bR, m \in \bN, 1 \leq q \leq \infty$, and $M_{\kappa,q} < M \leq \infty$. Let $T$ be in $\sio^{\kappa \pm}_{m,q,M}$ with kernel $K$ and $T^\ast$ be the adjoint of $T$ with respect to $L^2(\bR^{1+n}_+)$-duality. Then $T^\ast$ lies in $\sio^{\kappa \mp}_{m,q',M}$ with kernel $K^\ast$.
\end{prop}

\begin{proof}
We only prove the case assuming $T \in \sio^{\kappa+}_{m,q,M}$ with $\kappa \leq 0$. The other cases are left to the reader. The boundedness of $T^\ast$ is clear, and Corollary \ref{cor:sk_duality} shows $K^\ast \in \sk^\kappa_{m,q',M}$, so it suffices to prove the representation. For any $f \in L^2_{-\kappa}(\bR^{1+n}_+)$ with bounded support in $\overline{\bR^{1+n}_+}$ and $g \in L^2_\rmb(\bR^{1+n}_+)$ so that $\pi(g)$ is disjoint with $\pi(f)$, we get by \eqref{e:def_T-representation}
\[ \left \langle Tg,f \right \rangle_{L^2(\bR^{1+n}_+)} = \int_{\bR^{1+n}_+} \left( \int_0^s (K(s,t) g(t))(y) dt \right ) \ovf(s,y) dsdy. \]
Thanks to Lemma \ref{lemma:sk_norm_schur}, Schur test implies
\begin{align*}
    &\int_{\bR^{1+n}_+} \int_0^s |(K(s,t) g(t))(y)| |f(s,y)| dt dsdy \\
    &= \int_0^\infty ds \int_0^s dt \int_{\bR^n} |\I_{\pi(f)}(s,y) (K(s,t) g(t))(y)| |f(s,y)| dy \\
    &\leq \int_0^\infty ds \int_0^s \omega_{\pi(f),\pi(g)}(s,t) \|g(t)\|_2 \|f(s)\|_2 ~dt \\
    &\lesssim \|g\|_{L^2(\bR^{1+n}_+)} \|f\|_{L^2(\bR^{1+n}_+)} \lesssim \|g\|_{L^2(\bR^{1+n}_+)} \|f\|_{L^2_{-\kappa}(\bR^{1+n}_+)} < \infty,
\end{align*}
using that $f$ has bounded support. Fubini's theorem ensures that
\begin{align*}
    \left \langle g,T^\ast f \right \rangle_{L^2(\bR^{1+n}_+)} 
    &= \left \langle Tg,f \right \rangle_{L^2(\bR^{1+n}_+)} \\
    &= \int_0^\infty ds \int_0^s \left \langle K(s,t)g(t),f(s) \right \rangle_{L^2(\bR^n)} dt \\
    &= \int_0^\infty dt \int_t^\infty \left \langle g(t), K(s,t)^\ast f(s) \right \rangle_{L^2(\bR^n)} ds \\
    &= \int_{\bR^{1+n}_+} g(t,x) \left( \int_t^\infty \overline{(K(s,t)^\ast f(s))}(x) ds \right) dtdx.
\end{align*}
We conclude by arbitrariness of $g$ that for a.e. $(t,x) \in \pi(f)^c$,
\[ T^\ast(f)(t,x) = \int_t^\infty (K(s,t)^\ast f(s))(x) ds. \]
Hence \eqref{e:def_sio-_T-representation} holds for $T^\ast$ with kernel $K^\ast$.
\end{proof}
\section{Extensions of SIOs to tent spaces}
\label{sec:sio_tent}

The bounded extensions of SIOs of type $\pm (0,m,q,M)$ on tent spaces have been studied in \cite{auscher12maxreg}. In this section, we improve their results in our framework, simplify some proofs, and generalize to SIOs of type $\pm (\kappa,m,q,M)$ for $\kappa \neq 0$. 

\subsection{Basic properties of tent spaces}
\label{ssec:tent}

Readers can refer to \cite{Coifman-Meyer-Stein85TentSpaces} for the original definitions of tent spaces and proofs of the following facts. Some needed refinements can be found in \cite{Auscher2011Change_angle,amenta18weighted}.

For $0<p<\infty$, $m \in \bN$, $\beta \in \bR$, the \textit{tent space} $T^{p;m}_\beta$ consists of (possibly vector-valued, strongly) measurable functions $f$ on $\bR^{1+n}_+$ for which
\[ \|f\|_{T^{p;m}_\beta} := \left ( \int_{\bR^n} \left ( \int_0^\infty \fint_{B(x,t^{1/m})} |t^{-\beta} f(t,y)|^2 dtdy \right )^{p/2} dx \right )^{1/p} < \infty. \]
Our notation for weighted tent spaces is different from that in \cite{auscher12maxreg} and in fact corresponds to $T^{p,2;m}_{-\beta/2}$ there.

Note that $T^{2;m}_\beta$ can be identified with $L^2_\beta(\bR^{1+n}_+)$ for any $m \in \bN$ by Fubini's theorem. Let $L^2_\rmc(\bR^{1+n}_+)$ be the space of $L^2(\bR^{1+n}_+)$-functions with compact support in $\bR^{1+n}_+$. It is dense in $T^{p;m}_\beta$. In fact, for any $f \in T^{p;m}_\beta$ and compact subset $K \subset \bR^{1+n}_+$, $\I_K f$ lies in $L^2(\bR^{1+n}_+)$ with
\begin{equation}
    \|\I_K f\|_{L^2(\bR^{1+n}_+)} \eqsim_{p,m,\beta,K} \|\I_K f\|_{T^{p;m}_\beta} \le \|f\|_{T^{p;m}_\beta}.
    \label{e:tent_truncate_Lq}
\end{equation}

For any $p \in (1,\infty)$ and $\beta \in \bR$, the dual of $T^{p;m}_\beta$ can be identified with $T^{p';m}_{-\beta}$ via $L^2(\bR^{1+n}_+)$-duality. It also identifies the dual of $T^{p;m}_\beta$ with $T^{\infty;m}_{-\beta,([p,1])}$ when $p \leq 1$. The space $T^{\infty;m}_{-\beta,([p,1])}$ consists of measurable functions $f$ on $\bR^{1+n}_+$ for which
\[ \|f\|_{T^{\infty;m}_{-\beta,([p,1])}} := \sup_{B} \frac{1}{|B|^{[p,1]}} \left ( \int_0^{r(B)^m} \fint_B |t^{\beta} f(t,y)|^2 dtdy \right )^{1/2} < \infty,  \]
where $B$ describes all balls of $\bR^n$. 

Note that $L^2_\rmc(\bR^{1+n}_+)$ is not dense in $T^{\infty;m}_{-\beta,([p,1])}$ for any $p \le 1$ and $\beta \in \bR$ with respect to the topology induced by its norm. However, it is dense in the $\weakstar$ topology.
\begin{lemma}
    \label{lemma:L2c-Tinfty_weak*dense}
    For any $0<p \le 1$ and $\beta \in \bR$, $L^2_\rmc(\bR^{1+n}_+)$ is dense in $T^{\infty;m}_{-\beta,([p,1])}$ with respect to the weak${}^\ast$ topology as dual of $T^{p;m}_\beta$.
\end{lemma}

\begin{proof}
    Let $f$ be in $T^{\infty;m}_{-\beta,([p,1])}$. For any $R>1$, define $C_R:=[1/R,R] \times B(0,R)$ and $f_R:=f\I_{C_R}$.   For any $g \in T^{p;m}_{\beta}$,
    \[ \left| \langle f-f_R,g \rangle_{L^2(\bR^{1+n}_+)} \right| \le \int_{(C_R)^c} |f||g| \lesssim \|f\|_{ T^{\infty;m}_{-\beta,([p,1])}} \|\I_{(C_R)^c} g\|_{T^{p;m}_\beta}, \]
    which tends to 0 as $R \to \infty$ by dominated convergence. This proves the convergence of $f_R$  to $f$ as $R \to \infty$ for the $\weakstar$ topology.
\end{proof}

An important tool in the study of tent spaces is change of aperture. For any $\alpha>0$ and $x \in \bR^n$, the \textit{cone} $\Gamma_\alpha^m(x)$ with \textit{aperture} $\alpha$ and \textit{homogeneity} $m$ at $x$ is given by
\begin{equation}
    \Gamma_\alpha^m(x) := \{(t,y) \in \bR^{1+n}_+: |x-y|<\alpha t^{1/m}\}.
    \label{e:def_cone}
\end{equation}
Notice that the inner integral in the $T^{p;m}_\beta$-norm is an $L^2$-norm on $\Gamma_1^m(x)$. Changing aperture gives an equivalent norm. More precisely, define the \textit{conical square function} as
\begin{equation}
    \cA^{(\alpha)}_{\beta;m}(f)(x) := \left( \int_0^\infty dt \int_{ B(x,\alpha t^{1/m}) } |t^{-\beta} f(t,y)|^2 \frac{dy}{t^{n/m}} \right)^{1/2}.
    \label{e:def_A_concial_square}
\end{equation}
For $p \in (0,\infty)$, define 
\[ \|f\|_{T^{p;m}_{\alpha,\beta}} := \| \cA^{(\alpha)}_{\beta;m}(f) \|_p. \]
For $\alpha>0$ and $p \in (0,\infty)$, one has $T^{p;m}_\beta = T^{p;m}_{\alpha,\beta}$ and
\begin{equation}
    \label{e:tent_change_ap}
 \min\{ \alpha^{\frac{n}{2}},\alpha^{\frac{n}{p}} \} \|f\|_{T^{p;m}_{\beta}} \lesssim \|f\|_{T^{p;m}_{\alpha,\beta}} \lesssim \max\{ \alpha^{\frac{n}{2}},\alpha^{\frac{n}{p}} \} \|f\|_{T^{p;m}_{\beta}}. 
\end{equation}
We shall omit $\alpha$ if $\alpha=1$, and also write $\cA_{\beta;m} := \cA^{(1)}_{\beta;m}$.

For $0<p \leq 1$, a measurable function $a$ on $\bR^{1+n}_+$ is called a \textit{$T^{p;m}_\beta$-atom}, if there exists a ball $B \subset \bR^n$ such that $\supp(a) \subset [0,r(B)^m] \times B$, and
\[ \|a\|_{L^2_\beta(\bR^{1+n}_+)} \leq |B|^{-[p,2]}. \]
Such a ball $B$ is said to be \textit{associated} to $a$. Using the notion of atomic decomposition, the following result is a verbatim adaptation of \cite[Theorem 4.9]{Auscher-McIntosh-Russ2008_HardyRieMfd}, left to the reader.

\begin{lemma}
\label{lem:T-bdd-atom_then_extension}
    Let $\beta,\gamma \in \bR$ and $0<p \leq 1$. Let $T$ be a bounded linear operator from $L^2_\beta(\bR^{1+n}_+)$ to $L^2_{\gamma}(\bR^{1+n}_+)$. If $T$ is uniformly bounded from $T^{p;m}_\beta$-atoms to $T^{p;m}_{\gamma}$, then $T$ can be extended to a bounded linear operator from $T^{p;m}_\beta$ to $T^{p;m}_{\gamma}$.
\end{lemma}

\subsection{Main results}
\label{ssec:results_sio_tent}

In this section, we state our main results on the extension of SIOs and their adjoints to tent spaces $T^{p;m}_\beta$. This is divided into four statements depending on the type of SIOs and the range of $p$. Proofs are postponed to Section \ref{ssec:proof_sio_tent}.

Throughout the section, $\kappa \in \bR$ and $m \in \bN$ are fixed constants.

The first statement concerns the extension of operators in $\sio^{\kappa+}_{m,q,M}$ to $T^{p;m}_\beta$ for $p \leq 2$. Define 
\begin{equation}
    p_M:=\frac{2n}{n+2mM}, \quad p_q(\beta):= \frac{2nq}{2n+(2\beta+1)mq},
    \label{e:def_p_M,p_q}
\end{equation}
and
\begin{equation}
    M_c(\kappa,q):=\max\left\{ \frac{n}{2m},M_{\kappa,q} \right\},
    \label{e:def_M_c(kappa,q)}
\end{equation}
where $M_{\kappa,q}$ is defined in \eqref{e:def_M_kappa}. When $M>M_c(\kappa,q)$ then $M>M_{\kappa,q}$ so the class $\sio^{\kappa+}_{m,q,M}$ is well-defined, and also $p_M<1$.

\begin{prop}
\label{prop:sio_tent_p<2}
    Let $1 \leq q \leq 2, M_c(\kappa,q) < M \leq \infty$, and $\beta>-1/2$. Let $T$ be in $\sio^{\kappa+}_{m,q,M}$, and assume that $T$ is a bounded operator from $L^2_\beta(\bR^{1+n}_+)$ to $L^2_{\beta+\kappa}(\bR^{1+n}_+)$. 
    Then $T$ extends to a bounded operator from $T^{p;m}_\beta$ to $T^{p;m}_{\beta+\kappa}$ when $\max\{p_M,p_q(\beta)\}<p \leq 2$, if one of the following conditions holds:
    \begin{enumerate}
        \item 
        \label{case:large_q'_sio_tent_p<2}
        $q'>\frac{2n}{m(2\beta+1)}$ (or equivalently $p_q(\beta)<1$);

        \item
        \label{case:small_q'_sio_tent_p<2}
        {$2 \leq q'\leq \frac{2n}{m(2\beta+1)}$} (or equivalently $p_q(\beta) \geq 1$) and $M>\frac{n}{mq}$. 
    \end{enumerate}
\end{prop}

The next statement establishes the extension of operators in $\sio^{\kappa-}_{m,q,M}$ to $T^{p;m}_\beta$ when $p \leq 2$.

\begin{prop}
\label{prop:sio-_tent_p<2}
    Let $1 \leq q \leq 2, M_c(\kappa,q) < M \leq \infty$, and $\beta<1/2$. Let $T$ be in $\sio^{\kappa-}_{m,q,M}$, and assume that $T$ is a bounded operator from $L^2_{\beta-\kappa}(\bR^{1+n}_+)$ to $L^2_{\beta}(\bR^{1+n}_+)$.
    Then $T$ extends to a bounded operator from $T^{p;m}_{\beta-\kappa}$ to $T^{p;m}_{\beta}$ when $p_M < p \leq 2$.
\end{prop}

The third statement describes the extension of operators in $\sio^{\kappa+}_{m,q,M}$ on $T^{p;m}_\beta$ when $p \geq 2$.

\begin{cor}
\label{cor:sio_tent_p>2}
    Let $2 \leq q \leq \infty, M_c(\kappa,q) < M \leq \infty$, and $\beta>-1/2$. Let $T$ be in $\sio^{\kappa+}_{m,q,M}$, and assume that $T$ is a bounded operator from $L^2_{\beta}(\bR^{1+n}_+)$ to $L^2_{\beta+\kappa}(\bR^{1+n}_+)$. Then $T$ extends to a bounded operator 
    \begin{itemize}
        \item from $T^{p;m}_\beta$ to $T^{p;m}_{\beta+\kappa}$, when $2 \leq p \leq \infty$, 
        \item from $T^{\infty;m}_{\beta,([p,1])}$ to $T^{\infty;m}_{\beta+\kappa,([p,1])}$, when $p_M < p \leq 1$. 
    \end{itemize}
\end{cor}

The last statement is concerned with the extension of operators in $\sio^{\kappa-}_{m,q,M}$ on $T^{p;m}_\beta$ when $p \geq 2$.

\begin{cor}
    \label{cor:sio-_tent_p>2}
    Let $2 \leq q \leq \infty, M_c( \kappa, q) < M \leq \infty$, and $\beta<1/2$. Let $T$ be in $\sio^{\kappa-}_{m,q,M}$, and assume that $T$ is a bounded operator from $L^2_{\beta-\kappa}(\bR^{1+n}_+)$ to $L^2_{\beta}(\bR^{1+n}_+)$.
    \begin{enumerate}
        \item 
        \label{case:large_q_sio-_tent}
        If $q>\frac{2n}{m(2\beta+1)}$ (or equivalently $p_{q'}(\beta)<1$), then $T$ extends to a bounded operator
        \begin{itemize}
            \item from $T^{p;m}_{\beta-\kappa}$ to $T^{p;m}_{\beta}$, when $2 \leq p \leq \infty$, 
            \item from $T^{\infty;m}_{\beta-\kappa,([p,1])}$ to $T^{\infty;m}_{\beta,([p,1])}$, when $\max\{p_M,p_{q'}(\beta)\} < p \leq 1$.
        \end{itemize}

        \item 
        \label{case:small_q_sio-_tent}
        If {$2 \leq q \leq \frac{2n}{m(2\beta+1)}$} (or equivalently $p_{q'}(\beta) \geq 1$) and $M>\frac{n}{mq'}$, then $T$ extends to a bounded operator from $T^{p;m}_{\beta-\kappa}$ to $T^{p;m}_{\beta}$ when $2 \leq p < (p_{q'}(\beta))'$.
    \end{enumerate}
\end{cor}

\begin{remark}
    \label{rem:compareAKMP}
Let us compare with \cite{auscher12maxreg}.
    \begin{itemize}
        \item We treat the new cases $\kappa \ne 0$. The constant $M_c(\kappa,q)$ depends on $\kappa$, while the exponents $p_M$ and $p_q(\beta)$ do not.
        
        \item There is an improvement in Proposition~\ref{prop:sio_tent_p<2}~\eqref{case:small_q'_sio_tent_p<2} and hence Corollary~\ref{cor:sio-_tent_p>2}~\eqref{case:small_q_sio-_tent} when $q\ne 2$. The range of $p$ is larger  but at the expense of the extra decay on  $M$. This range for $p$ was obtained in the case of a maximal regularity operator when $\kappa=0$ (unpublished communication of Y. Huang). This extra decay only comes from our argument and we do not know how to remove it. But it should be the case, as a discontinuity on the required lower control for $M$ with respect to $q$ appears at $q'=\frac{2n}{m(2\beta+1)}$ by looking at the two cases.
    \end{itemize}
\end{remark}

\begin{remark}
    \label{rem:sio_tent_p<2_L2beta_ext_necessity}
    For $\kappa>0$, the bounded extensions of $\sio^{\kappa+}_{m,q,M}$-operators from $L^2_\beta(\bR^{1+n}_+)$ to $L^2_{\beta+\kappa}(\bR^{1+n}_+)$ are automatic, see Lemma \ref{lem:ext_L2_kappa>0}. For $\kappa = 0$, it was asserted in \cite[Theorem 2.2]{auscher12maxreg} but in fact the proof has a gap. It is unclear whether {the operators in $\sio^{0+}_{m,q,M}$ are bounded on $L^2_\beta(\bR^{1+n}_+)$}, so it needs to be assumed, and also for $\kappa<0$. However, it is true for maximal regularity operators, see \cite[Theorem 1.3]{Auscher-Axelsson2011_MR_L2beta}.
\end{remark}

\begin{remark} 
Both propositions are proved by density, while both corollaries are proved by duality arguments. See below. Usual arguments show that extension by duality agrees with extension by density in norm  when $p<\infty$ and for the $\weakstar$ topology when $p=\infty$.  
\end{remark}

\subsection{Proofs}
\label{ssec:proof_sio_tent}

In this section, we provide the proof of results in Section \ref{ssec:results_sio_tent}. The arguments follow those of \cite{auscher12maxreg} to which we will refer, except for Proposition \ref{prop:sio_tent_p<2} \eqref{case:small_q'_sio_tent_p<2}. For this one, we still follow the setup there but introduce a new method based on extrapolation of weighted estimates to improve the range of the exponent $p$. We first settle the proof of the last three results.

\begin{proof}[Proof of Proposition \ref{prop:sio-_tent_p<2}]
    The proof is a verbatim adaptation of that of \cite[Proposition 3.7]{auscher12maxreg}. We leave it to the reader.
\end{proof}

\begin{proof}[Proof of Corollary \ref{cor:sio_tent_p>2} and Corollary \ref{cor:sio-_tent_p>2}]
    The proofs are similar by a formal duality argument from Proposition \ref{prop:sio-_tent_p<2} and Proposition \ref{prop:sio_tent_p<2}, res-pectively. Therefore, we only prove the case $p_M < p \leq 1$ in Corollary \ref{cor:sio_tent_p>2}, and the rest is left to the reader. 
    
    Proposition \ref{prop:sio_duality} says $T^\ast \in \sio^{\kappa-}_{m,q',M}$, and by duality, $T^\ast$ extends to a bounded operator from $L^2_{-\beta-\kappa}(\bR^{1+n}_+)$ to $L^2_{-\beta}(\bR^{1+n}_+)$. Note that $M_c(\kappa,q)=M_c(\kappa,q')$, so for any $f \in T^{\infty;m}_{\beta+\kappa,([p,1])}$, Proposition \ref{prop:sio-_tent_p<2} implies that  $g \mapsto \left\langle f,T^\ast g \right\rangle$ is a bounded anti-linear functional on $T^{p;m}_{-\beta-\kappa}$. Define $Tf \in T^{\infty;m}_{\beta,([p,1])}$ as $Tf(g):=\left\langle f,T^\ast g \right\rangle$ via the identification, and we get
    \[ \|Tf\|_{T^{\infty;m}_{\beta+\kappa,([p,1])}} = \sup_{g:\|g\|_{T^{p;m}_{-\beta-\kappa}}=1} |\left\langle f,T^\ast g \right\rangle| \leq \|T^\ast\|_{\scrL(T^{p;m}_{-\beta-\kappa},T^{p;m}_{-\beta})} \|f\|_{T^{\infty;m}_{\beta,([p,1])}}, \] 
  where $\scrL(X,Y)$ is the space of bounded linear operators from the (quasi-)Banach space $X$ to the (quasi-)Banach space $Y$.
  
  This completes the proof.
\end{proof}

We now turn to the proof of Proposition \ref{prop:sio_tent_p<2}. We split the operator $T$ into two parts, the \textit{regular part}
\[ T_0(f)(t,y)=\int_0^{t/2} (K(t,s)f(s))(y) ds, \]
and the \textit{singular part}
\[ T_{\sing}(f)(t,y) = \int_{t/2}^t (K(t,s)f(s))(y) ds. \]
The meaning of the integrals  by taking the restricted singular kernels $\I_{\{s<t/2\}}(t,s) K(t,s)$ or $\I_{\{t/2<s<t\}}(t,s) K(t,s)$  is that of  \eqref{e:def_si+} if $\kappa>0$ or Lemma \ref{lem:sk_of_si_kappa<0} if $\kappa \leq 0$. 

Let us start with the singular part.

\begin{lemma}
\label{lem:sio_sing_ext_p<2}
    Let $T$ be in $\sio^{\kappa+}_{m,q,M}$ with $M>M_c(\kappa,q)$ and assume that $T$ is a bounded operator from $L^2_\beta(\bR^{1+n}_+)$ to $L^2_{\beta+\kappa}(\bR^{1+n}_+)$. Then, $T_{\sing}$ extends to a bounded operator from $T^{p;m}_\beta$ to $T^{p;m}_{\beta+\kappa}$ when $p_M < p \leq 2$.
\end{lemma}

\begin{proof}[Proof]
    By interpolation, it suffices to consider the extension of $T_{\sing}$ to $T^{p;m}_\beta$ for $p_M<p \leq 1$. By Lemma \ref{lem:T-bdd-atom_then_extension}, it suffices to prove that $T_{\sing}$ is uniformly bounded from $T^{p;m}_\beta$-atoms to $T^{p;m}_{\beta+\kappa}$. This is a verbatim adaptation of \cite[Lemma 3.4]{auscher12maxreg} once we assume that $T$ is bounded from $L^2_\beta(\bR^{1+n}_+)$ to $L^2_{\beta+\kappa}(\bR^{1+n}_+)$. We leave it to the reader.
\end{proof}

We then move on to estimate the regular part $T_0$. If $q'>\frac{2n}{m(2\beta+1)}$, cf. Proposition \ref{prop:sio_tent_p<2} \eqref{case:large_q'_sio_tent_p<2}, we also obtain the uniform boundedness of $T_0$ on $T^{p;m}_\beta$-atoms for $p_q(\beta) < p \leq 1$ and conclude by Lemma \ref{lem:T-bdd-atom_then_extension}. But it does not work if $q' \leq \frac{2n}{m(2\beta+1)}$, cf. Case \eqref{case:small_q'_sio_tent_p<2}. In this case, our strategy is to embed $T_0$ into the analytic family of operators $\{T_z\}_{z \in \bC}$ given by
\begin{equation}
    T_z(f)(t,y):= \int_0^{t/2} \left ( \frac{s}{t} \right )^z (K(t,s)f(s))(y) ds.
    \label{e:def_T_z}
\end{equation}
If $f \in L^2_\rmc(\bR^{1+n}_+)$, then the integral makes sense as a Bochner integral valued in $L^2(\bR^n)$ for any $z \in \bC$ and a.e. $(t,y) \in \bR^{1+n}_+$.

We study the boundedness of $T_0$ by Stein's interpolation. We show that if $\Re(z)$ is sufficiently large, we can still obtain a uniform bound of $T_z$ acting on atoms. The main task is to obtain boundedness when $\Re(z)<0$. This is given by the following proposition.

\begin{prop}
\label{prop:Tz_ext_p<2}
Let $K$ be in $\sk^\kappa_{m,q,M}$ with $1 \leq q < 2$ and $M>\frac{n}{mq}$. Then, for any $\beta \in \bR$ and $z \in \bC$ with $\Re(z)>-\beta-\frac{1}{2}$, $T_z$ can be extended to a bounded operator from $T^{p;m}_\beta$ to $T^{p;m}_{\beta+\kappa}$ when $q<p \leq 2$.
\end{prop}

Compared to \cite[Lemma 3.5]{auscher12maxreg}, this improves the range of $p$ from $p=2$ to $q<p \leq 2$. This is achieved via an extrapolation of weighted estimates technique. Let us start with a lemma.

\begin{lemma}
\label{lem:T_z_bdd_L2beta}
    Let $\{K(t,s)\}_{(t,s) \in \Delta^c}$ be a strongly measurable family of bounded operators on $L^2(\bR^n)$ such that \eqref{e:def_sk_bdd_K(t,s)_L2} holds for some $\kappa \in \bR$. Then, for any $\beta \in \bR$ and $z \in \bC$ with $\Re(z)>-\beta-\frac{1}{2}$, the integral \eqref{e:def_T_z} converges for $f \in L^2_\beta(\bR^{1+n}_+)$. {Moreover, it defines a bounded operator from $L^2_\beta(\bR^{1+n}_+)$ to $L^2_{\beta+\kappa}(\bR^{1+n}_+)$.}
\end{lemma}

\begin{proof}
    The proof is the same as that of \eqref{e:def_si+} but changing the kernel $k(t,s)$ to $\I_{\{0<s<t/2\}}(t,s) \left ( \frac{s}{t} \right )^{\Re(z)} (t-s)^{\kappa-1} s^{\beta+\frac{1}{2}} t^{-\beta-\kappa+\frac{1}{2}}$.
\end{proof}

\begin{remark}
    \label{rem:necessity_cond_L2_bdd}
    Lemma \ref{lem:T_z_bdd_L2beta} says $T_0$ is always bounded from $L^2_\beta(\bR^{1+n}_+)$ to $L^2_{\beta+\kappa}(\bR^{1+n}_+)$ if $\beta>-1/2$. However, it is unclear for $T_{\sing}$ if $\kappa \le 0$. This is why we need to assume it on $T$, or equivalently on $T_{\sing}$.
\end{remark}

Next, we bring in a pointwise estimate. For any $f \in L^2_{\loc}(\bR^{1+n}_+)$ and $x \in \bR^n$, define the \textit{vertical square function} as
\[ \cV_\beta(f)(x) := \left( \int_0^\infty |t^{-\beta} f(t,x)|^2 dt \right)^{1/2}. \]
Let $\cM$ be the centred Hardy--Littlewood maximal function given by
\[ \cM(g)(x):= \sup_{\rho>0} \fint_{B(x,\rho)} |g(y)| dy \]
for any $g \in L^1_{\loc}(\bR^n)$ and $x \in \bR^n$.

\begin{lemma}
\label{lemma:extra_pointwise_est}

With the assumptions of Proposition \ref{prop:Tz_ext_p<2}, it holds that 
\begin{equation}
    \cA_{\beta+\kappa;m}(T_z(f))(x) \lesssim \cV_\beta(\cM_q(f))(x)
    \label{e:extra_pointwise_est}
\end{equation}
for any $f \in L^2_\beta(\bR^{1+n}_+)$ and $x \in \bR^n$, where
\[ \cM_q(f)(t,x):=\sup_{\rho>0} \left ( \fint_{B(x,\rho)} |f(t,y)|^q dy \right )^{1/q} = \cM(|f(t)|^q) (x)^{1/q}. \]
\end{lemma}

\begin{proof}
From the definition  \eqref{e:def_A_concial_square}, we get
\begin{equation}
    \cA_{\beta+\kappa;m}(T_z(f))(x) \eqsim \left \| t^{-\frac{n}{2m}-\beta-\kappa} \|T_z(f)(t)\|_{L^2(B(x,t^{1/m}))} \right \|_{L^2(\bR_+,dt)}.
    \label{e:extra_pw_est_A(T_zf)}
\end{equation}
Fix $(t,x) \in \bR^{1+n}_+$. Define $B_0:=B(x,t^{1/m})$, $C_0:=2B_0$, and $C_j:=2^{j+1} B_0 \setminus 2^j B_0$ for $j \geq 1$, so that
\[ \| T_z(f)(t) \|_{L^2(B_0)} \leq \sum_{j \geq 0} \|\I_{B_0} T_z (\I_{C_j} f)(t)\|_2. \]
Using $L^q-L^2$ off-diagonal decay and $t-s \sim t$, we get
\begin{align*}
    \|\I_{B_0} T_z (\I_{C_j} f)(t)\|_2
    &\leq \int_0^{t/2} \left ( \frac{s}{t} \right )^{\Re(z)} \| \I_{B_0} K(t,s) (\I_{C_j} f(s))\|_2 ~ds \\
    &\lesssim 2^{-jmM} t^{-\Re(z)-1+\kappa-\frac{n}{m}[q,2]} \int_0^{t/2}  s^{\Re(z)} \|\I_{C_j} f(s)\|_q ~ds \\
    &\lesssim 2^{-j(mM-\frac{n}{q})} t^{-\Re(z)-1+\kappa+\frac{n}{2m}} \int_0^{t/2}  s^{\Re(z)} \cM_q(f(s))(x) ds,
\end{align*}
due to the fact that
\[ \|\I_{C_j} f(s)\|_q \leq \|\I_{2^{j+1}B_0} f(s)\|_q \lesssim_q (2^j t^{\frac{1}{m}})^{\frac{n}{q}} \cM_q(f(s))(x). \]
Since $M>\frac{n}{mq}$, taking the sum on $j$, we get
\[ \| T_z(f)(t) \|_{L^2(B_0)} \lesssim t^{-\Re(z)-1+\kappa+\frac{n}{2m}} \int_0^{t/2} s^{\Re(z)} \cM_q(f(s))(x) ds. \]
By \eqref{e:extra_pw_est_A(T_zf)}, we derive
\begin{align*}
    \cA_{\beta+\kappa;m}(T_z(f))(x)^2 
    &\lesssim \int_0^\infty \left ( \int_0^{t/2} \left ( \frac{s}{t} \right )^{\Re(z)+\beta+\frac{1}{2}} \cM_q( s^{-\beta+\frac{1}{2}} f(s))(x) \frac{ds}{s} \right )^2 \frac{dt}{t} \\
    &\lesssim \int_0^\infty (s^{-\beta} \cM_q(f(s))(x))^2 ds = \cV_\beta(\cM_q(f))(x)^2.
\end{align*}
The last inequality follows by Schur test, as 
\[ \int_0^{t/2} \left ( \frac{s}{t} \right )^{\Re(z)+\beta+\frac{1}{2}} \frac{ds}{s} \leq C(z,\beta), \quad \int_{2s}^\infty \left ( \frac{s}{t} \right )^{\Re(z)+\beta+\frac{1}{2}} \frac{dt}{t} \leq C(z,\beta),\]
when $\Re(z)>-\beta-\frac{1}{2}$.
\end{proof}

We can now prove Proposition \ref{prop:Tz_ext_p<2}. Terminologies and facts related to weights $w$ can be found in \cite{CruzUribe-Martell-P2016BanachSpaces}.

\begin{proof}[Proof of Proposition \ref{prop:Tz_ext_p<2}]
The conclusion when $p=2$ has been shown in Lemma \ref{lem:T_z_bdd_L2beta}. We turn to proving it for $p<2$.

For any weight $w \in A_{2/q}$ and $f \in L^2_\rmc(\bR^{1+n}_+)$, Lemma \ref{lemma:extra_pointwise_est} and the $L^{2/q}(w)$-boundedness of $\cM$ imply
\begin{align*}
\|\cA_{\beta+\kappa;m}(T_z(f))\|_{L^2(w)}^2 
&\lesssim \|\cV_\beta(\cM_q(f))\|_{L^2(w)}^2 \\ 
&= \int_0^\infty t^{-2\beta} dt \int_{\bR^n} \cM(|f(t)|^q)(x)^{\frac{2}{q}} w(x) dx \\ 
&\lesssim \int_0^\infty t^{-2\beta} dt \int_{\bR^n} |f(t,x)|^2 w(x) dx = \|\cV_\beta(f)\|_{L^2(w)}^2.
\end{align*}
Limited-range extrapolation (see \cite[Theorem 3.31]{CruzUribe-Martell-P2016BanachSpaces}) implies for any $p \in (q,2)$, $w \in A_{p/q} \cap \RH_{(2/p)'}$, and $f \in L^2_\rmc(\bR^{1+n}_+)$,
\[ \|\cA_{\beta+\kappa;m}(T_z(f))\|_{L^p(w)} \lesssim \|\cV_\beta(f)\|_{L^p(w)}. \]
As $p<2$, we invoke \cite[Proposition 2.3(b)]{auscher-hofmann-martell2010SquareFunctions} and get
\[ \|\cV_\beta(f)\|_{L^p(w)} \lesssim \|\cA_{\beta;m}(f)\|_{L^p(w)} \]
by change of scaling $t \to t^{1/m}$. Taking $w=1$, we obtain that for any $p \in (q,2)$ and $f \in L^2_\rmc(\bR^{1+n}_+)$,
\[ \|T_z(f)\|_{T^{p;m}_{\beta+\kappa}} \eqsim \|\cA_{\beta+\kappa;m}(T_z(f))\|_p \lesssim \|\cA_{\beta;m}(f)\|_p = \|f\|_{T^{p;m}_\beta}. \]
We conclude by a density argument. 
\end{proof}

Let us finish the proof of Proposition \ref{prop:sio_tent_p<2}.

\begin{proof}[Proof of Proposition \ref{prop:sio_tent_p<2}]
    The bounded extension of $T_{\sing}$  from $T^{p;m}_\beta$ to  $T^{p;m}_{\beta+\kappa}$ for $p_M<p \leq 2$ has been shown in Lemma \ref{lem:sio_sing_ext_p<2}. Consider the regular part $T_0$. \\

    \paragraph{Case \eqref{case:large_q'_sio_tent_p<2}: $q'>\frac{2n}{m(2\beta+1)}$}
    We obtain the uniform $T^{p;m}_{\beta+\kappa}$-bound of $T_0$ on $T^{p;m}_\beta$-atoms when $p_q(\beta) <p \leq 1$ by adapting the calculation in \cite[Lemma 3.5]{auscher12maxreg}. As the range of $q$ corresponds to $p_q(\beta)<1$, we conclude using Lemma \ref{lem:T-bdd-atom_then_extension}. \\
    
    \paragraph{Case \eqref{case:small_q'_sio_tent_p<2}: $q' \leq \frac{2n}{m(2\beta+1)}$}
    We first claim that $T_z$ extends to a bounded operator from $T^{1;m}_\beta$ to $T^{1;m}_{\beta+\kappa}$ if $\Re(z)>\frac{n}{mq'}-\beta-\frac{1}{2} \ge 0$. Indeed, the adaptation of \cite[Lemma 3.5]{auscher12maxreg} also implies the uniform $T^{1;m}_{\beta+\kappa}$-bound of $T_z$ on $T^{1;m}_\beta$-atoms, so the claim follows by Lemma \ref{lem:T-bdd-atom_then_extension}. 

    Then the discussion is organised into two cases.
    
    \begin{itemize}
        \item If $q'>2$, Proposition \ref{prop:Tz_ext_p<2} says $T_z$ extends to a bounded operator from $T^{p;m}_\beta$ to $T^{p;m}_{\beta+\kappa}$ if $\Re(z)>-\beta-\frac{1}{2}$ and $q< p \leq 2$.
        
       Then for any $p \in (p_q(\beta),2]$, one can find $t_0 \in (-\beta-\frac{1}{2},0)$, $t_1>\frac{n}{mq'}-\beta-\frac{1}{2}$, and $q_0 \in (p,2]$ so that
        \[ (1-\theta)t_0+\theta t_1 = 0, \quad \frac{1-\theta}{q_0}+\frac{\theta}{1} = \frac{1}{p} \]
        for some $\theta \in (0,1)$. For any $y \in \bR$, the above discussion implies $T_{t_0+iy} \in \scrL(T^{q_0;m}_\beta,T^{q_0;m}_{\beta+\kappa})$ and $T_{t_1+iy} \in \scrL(T^{1;m}_\beta,T^{1;m}_{\beta+\kappa})$ with their operator norms independent of $y$. Then $T_0$ extends to a bounded operator from $T^p_\beta$ to $T^p_{\beta+\kappa}$ as desired, thanks to Stein's interpolation theorem on tent spaces (see \cite{Harboure-Torrea-Viviani1991VectorValued}). 

        \item If $q'=2$, the adaptation of \cite[Lemma 3.5]{auscher12maxreg} shows that $T_z$ extends to a bounded operator from $T^{2;m}_\beta$ to $T^{2;m}_{\beta+\kappa}$ if $\Re(z)>-\beta-\frac{1}{2}$. We hence also get the bounded extension of $T_0$ from $T^p_\beta$ to $T^p_{\beta+\kappa}$ when $p \in (p_2(\beta),2]$ by the same argument as above. 
    \end{itemize}
    
    This completes the proof.
\end{proof}
\section{Inhomogeneous Cauchy problems}
\label{sec:inhomo}

We may now begin the study of well-posedness in weighted tent spaces of Cauchy problems for the inhomogeneous parabolic equation
\begin{equation}
    \partial_t u(t,x) - \Div_x(A(x) \nabla_x u)(t,x) = f(t,x), \quad (t,x) \in (0,\infty) \times \bR^n,
    \tag{IC}
    \label{e:def_ic}
\end{equation}
where $f \in \scrD'(\bR^{1+n}_+)$ is called the \textit{source term}. Assume that $A \in L^\infty(\bR^n;\mat_n(\bC))$ is \textit{uniformly elliptic}, \textit{i.e.}, there exist $\Lambda_0,\Lambda_1>0$ such that, for a.e. $x \in \bR^n$ and any $\xi,\eta \in \bC^n$,
\begin{equation}
    \Re(\langle A(x)\xi,\xi \rangle) \geq \Lambda_0 |\xi|^2, \quad |\langle A(x)\xi,\eta \rangle| \leq \Lambda_1 |\xi| |\eta|.
    \label{e:def_A_uniformly_ell}
\end{equation}

Let us recall the classical definition of weak solutions to \eqref{e:def_ic}. Let $0 \leq a < b \leq \infty$ be constants, $\Omega$ be an open subset of $\bR^n$, and $Q:=(a,b) \times \Omega$. Let $f$ be in $\scrD'(Q)$. We say that a function $u \in L^2_{\loc}((a,b);H^1_{\loc}(\Omega))$ is a \textit{weak solution to \eqref{e:def_ic} with source term $f$}, if
\begin{equation}
    -\int_Q u(t,x) \overline{\partial_t \phi}(t,x) + \int_Q (A(x) \nabla u(t,x)) \cdot \overline{\nabla \phi}(t,x) = (f,\overline{\phi})
    \label{e:def_weaksol_ic}
\end{equation}
holds for any $\phi \in C_\rmc^\infty(Q)$, where the right-hand side is understood as the pairing between distributions and test functions on $Q$. Here and often in the sequel, when unspecified, the measures in the integrals are unweighted Lebesgue measures. We say $u$ is a \textit{global weak solution to \eqref{e:def_ic} with source term $f$} when $Q=\bR^{1+n}_+=(0,\infty) \times \bR^n$.

We shall use an extension of Duhamel's formula to build weak solutions. Define the operator $L:=-\Div(A(x)\nabla)$ on $L^2(\bR^n)$ with domain
\[ D(L):=\{f \in H^1(\bR^n):\Div(A\nabla f) \in L^2(\bR^n)\}. \]
It is well-known that $-L$ is a densely-defined closed operator on $L^2(\bR^n)$ that generates a bounded analytic semigroup $(e^{-tL})_{t \geq 0}$ on $L^2(\bR^n)$. 

Let $1 \leq p \leq q \leq \infty$ and $M>0$ be constants. We say that a family of bounded operators $(T_t)_{t \geq 0}$ on $L^2(\bR^n)$ has \textit{(parabolic) $L^p-L^q$ off-diagonal estimates} of \textit{order} $M$, if there is a constant $C>0$ so that, for any Borel sets $E,F \subset \bR^n$, $f \in L^2(\bR^n) \cap L^p(\bR^n)$, and a.e. $t>0$,
\begin{equation}
    \|\I_E T_t \I_F f\|_q \leq Ct^{-\frac{n}{2}[p,q]} \left ( 1+\frac{\dist(E,F)^2}{t} \right )^{-M} \|\I_F f\|_p.
    \label{e:def_Tt_ode}
\end{equation}
Let $p_\pm(L), q_\pm(L) \in [1,\infty]$ be the critical numbers introduced by \cite{Auscher2007Memoire}. The semigroup $(e^{-tL})_{t \geq 0}$ has $L^p-L^q$ off-diagonal estimates of any order $M$  when $p_-(L)<p \leq q<p_+(L)$, and the same holds for the family $(t^{1/2} \nabla e^{-tL})_{t>0}$ when $q_-(L)<p \leq q<q_+(L)$. Moreover, $p_-(L)=q_-(L)<2$ and $p_+(L) \geq q_+(L)>2$. \\

Let us also recall a family of slice spaces. Readers can refer to \cite[\S 3]{Auscher-Mourgoglou2019Slice} for the proof of the following facts. For any $p \in [1,\infty]$ and $\delta>0$, the \textit{(parabolic) slice space $E^p_\delta$} consists of measurable functions $f$ for which
\[ \|f\|_{E^p_\delta} := \left\| \left ( \fint_{B(\cdot,\delta^{1/2})} |f(y)|^2 dy \right )^{1/2} \right \|_p < \infty. \]
Let $\delta,\delta'>0$ be constants. Then, for any measurable function $f$ on $\bR^n$,
\begin{equation}
    \|f\|_{E^p_\delta} \lesssim \max \left \{ 1,\left ( \frac{\delta'}{\delta}  \right )^{\frac{n}{2}[2,p]} \right \} \|f\|_{E^p_{\delta'}}.
    \label{e:Epdelta_change}
\end{equation}
Fix $\delta>0$. For any $p \in [1,\infty)$, the dual of $E^p_\delta$ can be identified with $E^{p'}_\delta$ via $L^2(\bR^n)$-duality. Moreover, let $1 \le p_0<p<p_1 \leq \infty$ be with $1/p=(1-\theta)/p_0+\theta/p_1$ for some $\theta>0$. Then, the complex interpolation space $[E^{p_0}_\delta,E^{p_1}_\delta]_\theta$ can be identified with $E^p_\delta$.

\subsection{Main results}
\label{ssec:results_ic}

Define the operator $\cL_1$ on $L^2(\bR^{1+n}_+)$ by
\begin{equation}
    \cL_1(f)(t,y) := \int_0^t (e^{-(t-s)L}f(s))(y) ds.
    \label{e:def_L1}
\end{equation}
For any $f \in L^2(\bR^{1+n}_+)$, $\cL_1(f)$ and $\nabla \cL_1(f)$ are all well-defined in $L^2_{\loc}(\bR^{1+n}_+)$. The first proposition concerns \textit{a priori} estimates in tent spaces. Introduce 
\begin{equation}
    p_L(\beta):=\frac{np_-(L)}{n+(2\beta+1)p_-(L)},
    \label{e:def_pL(beta)}
\end{equation}
which agrees with $p_q(\beta)$ defined in \eqref{e:def_p_M,p_q} for $q=p_-(L)$ and $m=2$. From now on, we fix $m=2$ and write $T^p_{\beta}:=T^{p;2}_{\beta}$.

\begin{prop}[\textit{A priori} estimates]
    \label{prop:est_f_L2c}
    Let $\beta>-1/2$ and $p_L(\beta)<p \leq \infty$. Let $f$ be in $L^2_\rmc(\bR^{1+n}_+)$.
    \begin{enumerate}[(i)]
        \item $\cL_1(f)$ lies in $T^p_{\beta+1}$ and $\nabla \cL_1(f)$ lies in $T^p_{\beta+\frac{1}{2}}$ with
        \[ \|\cL_1(f)\|_{T^{p}_{\beta+1}} \lesssim \|f\|_{T^{p}_\beta}, \quad \|\nabla \cL_1(f) \|_{T^{p}_{\beta+\frac{1}{2}}} \lesssim \|f\|_{T^{p}_\beta}. \]
        \label{item:ic_regularity_L2c}

        \item For any $\phi \in C_\rmc^\infty(\overline{\bR^{1+n}_+})$,
        \begin{equation}
            \int_{\bR^{1+n}_+} |\cL_1(f)| |\partial_t \phi| + \int_{\bR^{1+n}_+} |A\nabla \cL_1(f)| |\nabla \phi| + \int_{\bR^{1+n}_+} |f||\phi| \lesssim_{\phi} \|f\|_{T^p_\beta},
            \label{e:L1f_weak_make_sense_L2c}
        \end{equation}
        and
        \begin{equation}
            -\int_{\bR^{1+n}_+} \cL_1(f) \overline{\partial_t \phi} + \int_{\bR^{1+n}_+} (A\nabla \cL_1(f)) \cdot \overline{\nabla \phi}  = \int_{\bR^{1+n}_+} f\overline{\phi}.
            \label{e:L1f_weak_sol_L2c}
        \end{equation}
        \label{item:ic_sol_L2c}
    \end{enumerate}
\end{prop}

The proof is postponed to Section \ref{ssec:a_priori_est_ic}. Obviously, the proposition implies that $\cL_1$ can be extended to an operator acting on $T^p_\beta$. We summarize the properties of the extension in the following theorem. For any $(t,x) \in \bR^{1+n}_+$, let $W(t,x):=(t,2t) \times B(x,t^{1/2})$ be a \textit{(parabolic) Whitney cube} at $(t,x)$.

\begin{theorem}[Existence]
    \label{thm:ic_sol}
    Let $\beta>-1/2$ and $p_L(\beta)<p \leq \infty$. There is an extension of $\cL_1$ defined by \eqref{e:def_L1} on $L^2_\rmc(\bR^{1+n}_+)$ to a bounded operator from $T^p_\beta$ to $T^p_{\beta+1}$, also denoted by $\cL_1$, such that the following properties hold when $f \in T^p_\beta$ and $u:=\cL_1(f)$.
    \begin{enumerate}[(a)]
        \item (Regularity) $u$ lies in $T^{p}_{\beta+1}$ and $\nabla u$ lies in $T^{p}_{\beta+\frac{1}{2}}$ with
        \[ \|u\|_{T^{p}_{\beta+1}} \lesssim \|f\|_{T^{p}_\beta}, \quad \|\nabla u\|_{T^{p}_{\beta+\frac{1}{2}}} \lesssim \|f\|_{T^{p}_\beta}. \]
        \label{item:ic_regularity}
        
        \item For any $\phi \in C_\rmc^\infty(\overline{\bR^{1+n}_+})$,
        \begin{equation}
            \int_{\bR^{1+n}_+} |u| |\partial_t \phi| + \int_{\bR^{1+n}_+} |A\nabla u| |\nabla \phi| + \int_{\bR^{1+n}_+} |f||\phi| \lesssim_\phi \|f\|_{T^p_\beta},
            \label{e:L1f_weak_make_sense}
        \end{equation}
        and
        \begin{equation}
            -\int_{\bR^{1+n}_+} u \overline{\partial_t \phi} + \int_{\bR^{1+n}_+} (A\nabla u) \cdot \overline{\nabla \phi}  = \int_{\bR^{1+n}_+} f\overline{\phi}.
            \label{e:L1f_weak_sol}
        \end{equation}
        Therefore, $u$ is a global weak solution to \eqref{e:def_ic} with source term $f$.
        \label{item:ic_sol}

        \item (Maximal regularity) $\partial_t u$ and $\Div(A\nabla u)$ belong to $T^{p}_{\beta}$ with
        \[ \|\partial_t u\|_{T^{p}_{\beta}} + \|\Div(A\nabla u)\|_{T^{p}_{\beta}} \lesssim \|f\|_{T^{p}_\beta}. \]
        \label{item:ic_max_reg}

        \item (Whitney trace) For a.e. $x \in \bR^n$,
        \[ \lim_{t \rightarrow 0} \left( \fint_{W(t,x)} |u(s,y)|^2 dsdy \right)^{1/2} = 0. \]
        \label{item:ic_boundary_whitney}

        \item (Distributional trace) $u(t)$ converges to 0 as $t \to 0$ in $L^p(\bR^n)$ if $p_L(\beta)<p \leq 2$, and in $E^q_\delta$ for any $\delta>0, q \in [p,\infty]$ if $2<p \leq \infty$.
        \label{item:ic_boundary_LpEp}
    \end{enumerate}
\end{theorem}
Statements \eqref{item:ic_regularity} and \eqref{item:ic_sol} are proved in Section \ref{ssec:exist_reg_ic}, and the others are proved in Section \ref{ssec:boundary_ic}. Let us first make some remarks.

In \eqref{item:ic_sol}, not only do we prove that $\cL_1(f)$ is a global weak solution to \eqref{e:def_ic} with source term $f$, but we also implicitly imply $\cL_1(f)(0)=0$ as the test functions are arbitrary at $t=0$. In turn, \eqref{item:ic_boundary_whitney} and \eqref{item:ic_boundary_LpEp} explicitly show the null \textit{boundary behavior} of $\cL_1(f)(t)$ as $t \to 0$. As we shall see in Lemma \ref{lemma:whitney-boundary_tent}, \eqref{item:ic_boundary_whitney} in fact holds for any function $u \in T^{p}_{\beta+1}$. 

Furthermore, \eqref{item:ic_sol} yields the equality $\partial_t u-\Div(A \nabla u)=f$ holds in $\scrD'(\bR^{1+n}_+)$, but \eqref{item:ic_max_reg} implies it actually holds in $T^p_\beta$. In fact, $\partial_t u$ and $\Div(A\nabla u)$ cannot lie in a better space simultaneously, otherwise the regularity of $f$ could be raised. It is exactly the spirit of maximal regularity although the interpretation is no longer via semigroup theory.

Then, we turn to uniqueness. 
\begin{theorem}[Uniqueness]
    \label{thm:unique_ic}
    Let $\beta>-1/2$ and $p_L(\beta)<p \leq \infty$. Then, for any $f \in T^{p}_\beta$, there exists at most one global weak solution $u \in T^{p}_{\beta+1}$ to \eqref{e:def_ic} with source term $f$.
\end{theorem}

The proof is deferred to Section \ref{ssec:uniquess_ic}. Altogether, we have demonstrated the main theorem, Theorem \ref{thm:main}.

\begin{remark}
    Local well-posedness also holds, as the proofs will show that one can work on any given strip $[0,T] \times \bR^n$ instead. 
\end{remark}

\subsection{\textit{A priori} estimates}
\label{ssec:a_priori_est_ic}
In this section, we prove Proposition \ref{prop:est_f_L2c}. Define the operator $\cL_{1/2}: L^2(\bR^{1+n}_+) \to L^1_{\loc}((0,\infty);L^2(\bR^n))$ by
\begin{equation}
    \cL_{1/2}(f)(t,y) := \int_0^t (\nabla e^{-(t-s)L}f(s))(y) ds, \quad \mathrm{a.e.}\ (t,y)\in \bR^{1+n}_+.
    \label{e:def_op_reg_inhomo}
\end{equation}
The integral exists as an $L^2(\bR^n)$-valued Bochner integral. 

\begin{lemma}
    \label{lemma:L1/2=nabla_L1}
    For any $f \in L^2_\rmc(\bR^{1+n}_+)$,
    \begin{equation}
        \cL_{1/2}(f) = \nabla \cL_1(f) \quad \text{ in } \scrD'(\bR^{1+n}_+).
        \label{e:L1/2=nabla_L1}
    \end{equation}
    In particular, \eqref{e:L1f_weak_sol_L2c} holds if and only if for all $\phi \in C_\rmc^\infty(\overline{\bR^{1+n}_+})$,
    \begin{equation}
        -\int_{\bR^{1+n}_+} \cL_1(f) \overline{\partial_t \phi} + \int_{\bR^{1+n}_+} A\cL_{1/2}(f) \cdot \overline{\nabla \phi} = \int_{\bR^{1+n}_+} f\overline{\phi}.
        \label{e:ic_weak_sol_by_L1/2}
    \end{equation}
\end{lemma}
\begin{proof}
    It suffices to prove \eqref{e:L1/2=nabla_L1}. Indeed, if it holds, then $\nabla \cL_1(f)=\cL_{1/2}(f)$ in $L^1_{\loc}((0,\infty);L^2(\bR^n))$, so the equivalence follows obviously. 
    
    Let us prove \eqref{e:L1/2=nabla_L1}. Let $\psi \in C_\rmc^\infty(\bR^{1+n}_+;\bC^n)$. Fubini's theorem yields
    \begin{align*}
    \int_{\bR^{1+n}_+} \cL_{1/2}(f) \cdot \overline{\psi}
     &= \int_0^\infty dt \int_0^t \langle \nabla e^{-(t-s)L} f(s), \psi(t) \rangle_{L^2(\bR^n;\bC^n)} ds \\ 
     &= -\int_0^\infty dt \int_0^t \langle e^{-(t-s)L} f(s), \Div \psi(t) \rangle_{L^2(\bR^n)} ds \\
     &= -\int_{\bR^{1+n}_+} \cL_1(f) \Div \overline{\psi}.
    \end{align*}
    This completes the proof.
\end{proof}

We also need an embedding lemma. 
\begin{lemma}
    \label{lem:i-c_embed_test_tent}
    Let $m \in \bN$ and $\beta>-1/2$. Then,  $L^\infty_\rmc(\overline{\bR^{1+n}_+})$ embeds into $T^{p;m}_{-\beta}$ and $T^{\infty;m}_{-\beta;(\sigma)}$ for $p \in (0,\infty)$ and $0 \leq \sigma \leq \frac{m}{n}(\beta+\frac{1}{2})$.
\end{lemma}

\begin{proof}
    Let $\phi$ be an $L^\infty_\rmc(\overline{\bR^{1+n}_+})$-function supported in $[0,R] \times B(0,R^{1/m})$ for some $R>0$. Then $\cA_{-\beta;m}(\phi)$ is supported on $B(0,2R^{1/m})$, and for any $x \in B(0,2R^{1/m})$,
    \[ \cA_{-\beta;m}(\phi)(x)^2 = \int_0^R dt \fint_{B(x,t^{1/m})} |t^{\beta} \phi(t,y)|^2 dy \leq \|\phi\|_\infty^2 \int_0^R t^{2\beta} dt \lesssim \|\phi\|_\infty^2 \]
    as $\beta>-1/2$. Thus, for any $p \in (0,\infty)$,
    \[ \|\phi\|_{T^{p;m}_{-\beta}} = \|\cA_{-\beta;m}(\phi)\|_p \lesssim_{R,\beta} \|\phi\|_\infty. \]
    Similarly, for any ball $B=B(x_0,r^{1/m}) \subset \bR^n$, we get
    \begin{align*}
        &\frac{1}{|B|^\sigma} \left ( \int_0^r \I_{[0,R]}(t) dt \fint_B |t^{\beta} \phi(t,y)|^2 dy \right )^{1/2} \\
        &\quad \lesssim r^{-\frac{n}{m}\sigma} \left ( \int_0^{\min\{r,R\}} t^{2\beta} dt  \right )^{1/2} \|\phi\|_\infty \lesssim R^{\beta+\frac{1}{2}-\frac{n}{m}\sigma} \|\phi\|_\infty.
    \end{align*}
    We hence conclude by taking the supremum over all balls $B$, since the controlling constant is independent of $B$.
\end{proof}

We now prove Proposition \ref{prop:est_f_L2c}.
\begin{proof}[Proof of Proposition \ref{prop:est_f_L2c}]
    First, observe that
    \begin{equation}
        \begin{cases}
            \cL_1 \in \sio^{1+}_{2,q,\infty} & \text{ if } p_-(L) < q < p_+(L), \\ 
            \cL_{1/2} \in \sio^{\frac{1}{2}+}_{2,q,\infty} & \text{ if } q_-(L) < q < q_+(L).
        \end{cases}
        \label{e:L1-1/2_sio}
    \end{equation}
    Indeed, it suffices to consider the former assertion, and the latter follows similarly. When $p_-(L)<q \leq r<p_+(L)$, the family $(e^{-(t-s)L})_{t>s}$ has $L^q-L^r$ off-diagonal decay of type $(1,2,M)$ for any $M>0$. Thus, for $q \in (p_-(L),p_+(L))$, the function $(t,s) \mapsto \I_{\{t>s\}}(t,s) e^{-(t-s)L}$ belongs to $\sk^1_{2,q,\infty}$. In consequence, Lemma \ref{lem:ext_L2_kappa>0} implies $\cL_1 \in \sio^{1+}_{2,q,\infty}$.
    
    Next, we prove \eqref{item:ic_regularity_L2c}. Since $L^2_\rmc(\bR^{1+n}_+)$ is contained in $T^p_\beta$, applying \eqref{e:L1-1/2_sio}, Proposition \ref{prop:sio_tent_p<2}, Corollary \ref{cor:sio_tent_p>2}, and Lemma \ref{lemma:L1/2=nabla_L1}, we get \eqref{item:ic_regularity_L2c} as 
    \[ \|\cL_1(f)\|_{T^p_{\beta+1}} \lesssim \|f\|_{T^p_\beta}, \quad \|\nabla \cL_1(f)\|_{T^p_{\beta+\frac{1}{2}}} = \|\cL_{1/2}(f)\|_{T^p_{\beta+\frac{1}{2}}} \lesssim \|f\|_{T^p_\beta}. \]
    
    Then consider \eqref{item:ic_sol_L2c}. Fix $\phi \in C_\rmc^\infty(\overline{\bR^{1+n}_+})$. For \eqref{e:L1f_weak_make_sense_L2c}, Lemma \ref{lem:i-c_embed_test_tent} yields $\phi \in (T^p_\beta)'$, $\nabla \phi \in (T^p_{\beta+\frac{1}{2}})'$, and $\partial_t \phi \in (T^p_{\beta+1})'$. Using duality of tent spaces and \eqref{item:ic_regularity_L2c}, we get
    \[ \lhs \text{ \eqref{e:L1f_weak_make_sense_L2c}} \lesssim_\phi \|\cL_1(f)\|_{T^p_{\beta+1}} + \|\nabla \cL_1(f)\|_{T^p_{\beta+\frac{1}{2}}} + \|f\|_{T^p_\beta} \lesssim \|f\|_{T^p_\beta}. \]
    This proves \eqref{e:L1f_weak_make_sense_L2c}. For \eqref{e:L1f_weak_sol_L2c}, the proof is divided into two cases.

    \
    
    \paragraph{Case 1: $f \in L^2_\rmc((0,\infty);D(L))$}
    By definition of $\cL_1$ (cf.~\eqref{e:def_L1}), we have $\cL_1(f) \in C([0,\infty);L^2(\bR^n))$ with $\cL_1(f)(0)=0$, and $\cL_1(f)(t) \in D(L)$ for all $t>0$. Thus, we infer $L \cL_1(f) \in L^2(\bR^{1+n}_+)$, so
    \begin{equation}
        \partial_t \cL_1(f) = f + \Div(A\nabla \cL_1(f)) = f - L\cL_1(f)  \in L^2(\bR^{1+n}_+).
        \label{e:eq_L2c_case}
    \end{equation}
    Then \eqref{e:L1f_weak_sol_L2c} holds by testing $\phi \in C_\rmc^\infty(\bR^{1+n}_+)$ against \eqref{e:eq_L2c_case} and using integration by parts.

    \
    
    \paragraph{Case 2: $f \in L^2_\rmc(\bR^{1+n}_+)$} Lemma \ref{lemma:L1/2=nabla_L1} says that  it is equivalent to prove \eqref{e:ic_weak_sol_by_L1/2}. 
    Lemma \ref{lem:i-c_embed_test_tent} implies $\phi \in L^2(\bR^{1+n}_+)$, $\nabla \phi \in L^2_{-1/2}(\bR^{1+n}_+)$, and $\partial_t \phi \in L^2_{-1}(\bR^{1+n}_+)$. Thus, \eqref{e:ic_weak_sol_by_L1/2} readily follows by a standard density argument using Case 1 and continuity of $\cL_1:L^2(\bR^{1+n}_+) \rightarrow L^2_1(\bR^{1+n}_+)$ and $\cL_{1/2}:L^2(\bR^{1+n}_+) \rightarrow L^2_{1/2}(\bR^{1+n}_+)$. This completes the proof.
\end{proof}

\subsection{Existence}
\label{ssec:exist_reg_ic}
In this section, we construct the extension of $\cL_1$, and prove Theorem \ref{thm:ic_sol} \eqref{item:ic_regularity} and \eqref{item:ic_sol}.

\begin{lemma}
    \label{lemma:ext_L1-1/2_tent}
    Let $\beta>-1/2$ and $p_L(\beta)<p \leq \infty$. Then, for $\kappa=1$ and $\kappa=1/2$, $\cL_\kappa$ extends to bounded operators from $T^{p}_\beta$ to $T^{p}_{\beta+\kappa}$. Furthermore, \eqref{e:L1/2=nabla_L1} holds for the extensions applied to $f \in T^{p}_\beta$.
\end{lemma}

\begin{proof}
    Using \eqref{e:L1-1/2_sio}, the extension follows by applying Proposition \ref{prop:sio_tent_p<2} and Corollary \ref{cor:sio_tent_p>2}. In the following, the extension of $\cL_\kappa$ to tent spaces is still denoted by $\cL_\kappa$. We next divide the proof of \eqref{e:L1/2=nabla_L1} in two cases. \\ 

    \paragraph{Case 1: $p_L(\beta)<p<\infty$}
    Lemma \ref{lemma:L1/2=nabla_L1} says that  it holds when $f \in L^2_\rmc(\bR^{1+n}_+)$. A standard density argument extends it to $f\in T^{p}_\beta$, ensured by three ingredients: the density of $L^2_\rmc(\bR^{1+n}_+)$ in $T^p_\beta$; the continuity of $\cL_\kappa$ from $T^p_\beta$ to $T^{p}_{\beta+\kappa}$ for $\kappa=1,1/2$;  and the fact that any $C_\rmc^\infty({\bR^{1+n}_+})$-function belongs to $(T^p_{\beta+1})' \cap (T^p_{\beta+\frac{1}{2}})'$ (cf. Lemma \ref{lem:i-c_embed_test_tent}). \\ 

    \paragraph{Case 2: $p=\infty$}
    Let $\cL_\kappa^\ast$ be the adjoint of $\cL_\kappa$ with respect to $L^2(\bR^{1+n}_+)$-duality for $\kappa=1,1/2$. Lemma \ref{lemma:L2c-Tinfty_weak*dense} shows $L^2_\rmc(\bR^{1+n}_+)$ is $\weakstar$ dense in $T^\infty_\beta$, so the desired distribution equality follows by the continuity of $\cL_\kappa^\ast:T^1_{-\beta-\kappa} \to T^1_{-\beta}$ for $\kappa=1,1/2$, and the fact that $\Cc({\bR^{1+n}_+})$ is contained in $T^1_{-\beta-1}\cap T^1_{-\beta-\frac{1}{2}}$. This completes the proof.
\end{proof}

We can now prove Theorem \ref{thm:ic_sol} \eqref{item:ic_regularity} and \eqref{item:ic_sol}.

\begin{proof}[Proof of Theorem \ref{thm:ic_sol} \eqref{item:ic_regularity} and \eqref{item:ic_sol}]
    The extension is given by Lemma \ref{lemma:ext_L1-1/2_tent}. It also implies that if $u =  \cL_1(f)$ with $f\in T^p_\beta$, then $\nabla u =  \cL_{1/2}(f)$, so \eqref{item:ic_regularity} follows by the boundedness of $\cL_1$ and $\cL_{1/2}$.
    
    Then consider \eqref{item:ic_sol}. By \eqref{e:tent_truncate_Lq}, it is clear that both of $u$ and $\nabla u$ belong to $L^2_{\loc}(\bR^{1+n}_+)$, so $u \in L^2_{\loc}((0,\infty);H^1_{\loc}(\bR^n))$. The inequality \eqref{e:L1f_weak_make_sense} is also clear by Lemma \ref{lem:i-c_embed_test_tent} and \eqref{item:ic_regularity}, so it only remains to prove \eqref{e:L1f_weak_sol}. Here, we can follow the proof of Proposition \ref{prop:est_f_L2c} \eqref{item:ic_sol_L2c}, that is to say,
    \begin{itemize}
        \item Interpreting \eqref{e:L1f_weak_sol} as \eqref{e:ic_weak_sol_by_L1/2} using duality of tent spaces, since $\nabla \cL_1(f)$ agrees with $\cL_{1/2}(f)$ in $T^p_{\beta+\frac{1}{2}}$ and $\nabla \phi \in (T^p_{\beta+\frac{1}{2}})'$;
        \item Showing \eqref{e:ic_weak_sol_by_L1/2} by density with a sequence in $L^2_\rmc(\bR^{1+n}_+)$ that approximates $f$ in $T^p_\beta$ (or ${\weakstar}$ if $p=\infty$, see Case 2 of Lemma \ref{lemma:ext_L1-1/2_tent}).
    \end{itemize}
    All the ingredients have been provided in Lemma \ref{lem:i-c_embed_test_tent} and \ref{lemma:ext_L1-1/2_tent} and  details are  left to the reader.
\end{proof}

\subsection{Uniqueness}
\label{ssec:uniquess_ic}
In this section, we demonstrate Theorem \ref{thm:unique_ic} using the homotopy identity. We start with several preliminary lemmas.

\begin{lemma}
    \label{lemma:embed_tent_L2beta(a,b)_p<2}
    Let $0<p \leq 2$ and $\beta \in \bR$. For $0<a<b<\infty$, $T^{p}_\beta$ can be embedded into $L^2_\beta((a,b) \times \bR^n) := L^2((a,b),t^{-2\beta} dt;L^2(\bR^n))$ with
    \[ \|u\|_{L^2_\beta((a,b) \times \bR^n) } \lesssim_{p} a^{-\frac{n}{2p}} b^{\frac{n}{4}} \|u\|_{T^{p}_\beta}. \]
\end{lemma}

\begin{proof}
    It suffices to prove the case $\beta=0$. For any $R>a^{1/2}$, observe that $(a,b) \times B(0,R) \subset \Gamma_\lambda(x)$ for any $x \in B(0,R)$ with $\lambda:=2Ra^{-1/2}>1$, where $\Gamma_\lambda(x)$ is the cone at $x$ of aperture $\lambda$ and homogeneity $m=2$, given by \eqref{e:def_cone}. Then applying \eqref{e:tent_change_ap} for $p \leq 2$ yields
    \begin{align*}
    \|u\|_{L^2((a,b) \times B(0,R))}
     &\leq b^{n/4} \left ( \int_a^b \int_{B(0,R)} |u(t,y)|^2 \frac{dt}{t^{n/2}} dy \right )^{1/2} \\ 
     &\leq b^{n/4} \inf_{x \in B(0,R)} \cA^{(\lambda)}_{0;2}(u)(x) \\
     &\leq b^{n/4} \left ( \fint_{B(0,R)} \cA^{(\lambda)}_{0;2}(u)(x)^p dx \right )^{1/p} \\
     &\lesssim b^{\frac{n}{4}} R^{-\frac{n}{p}} \|u\|_{T^{p;2}_{\lambda,0}} \lesssim b^{\frac{n}{4}} R^{-\frac{n}{p}} \lambda^{\frac{n}{p}} \|u\|_{{T^{p}_0}} \lesssim a^{-\frac{n}{2p}} b^{\frac{n}{4}} \|u\|_{T^{p}_0},
    \end{align*}
    where $\cA^{(\lambda)}_{0;2}$ is defined in \eqref{e:def_A_concial_square}. The controlling constant is independent of $R$, so we conclude by letting $R$ tend to infinity.
\end{proof}

\begin{lemma}
    \label{lemma:homotopy-condition}
    Let $0< p \leq \infty$ and $\beta \in \bR$. Let $u$ be in $T^{p}_{\beta}$. Then for $0<a<b<\infty$, the function 
    \[ F(x):=\left ( \int_a^b \int_{B(x,b^{1/2})} |u(t,y)|^2 dtdy \right )^{1/2} \]
    belongs to $L^p(\bR^n)$ with $\|F\|_p \lesssim_{a,b,\beta} \|u\|_{T^{p}_{\beta}}$.
\end{lemma}

\begin{proof}
    First consider $0< p < \infty$. Pick $(z_k)_{1 \leq k \leq N}$ in $B(0,b^{1/2})$ such that $B(0,b^{1/2}) \subset \bigcup_{k=1}^N B(z_k,a^{1/2})$. Thus, for any $x \in \bR^n$, $B(x,b^{1/2}) \subset \bigcup_{k=1}^N B(x+z_k,a^{1/2})$. We hence get
    \begin{align*}
    F(x)
    &\leq \sum_{k=1}^N b^{\frac{n}{4}} \max\{a^\beta,b^\beta\} \left ( \int_a^b \int_{B(x+z_k,a^{1/2})} |t^{-\beta}u(t,y)| ^2 \frac{dt}{t^{n/2}}dy \right )^{1/2} \\ 
    &\le b^{\frac{n}{4}} \max\{a^\beta,b^\beta\} \sum_{k=1}^N \left ( \int_a^b \int_{B(x+z_k,t^{1/2})} |t^{-\beta}u(t,y)| ^2 \frac{dt}{t^{n/2}}dy \right )^{1/2}.
\end{align*}
Therefore, by taking $L^p$-norms, we conclude that
\begin{equation}
    \|F\|_p \lesssim_N b^{\frac{n}{4}} \max\{a^\beta,b^\beta\} \|u\|_{T^{p}_{\beta}}.
    \label{e:L2loc_tent_precise}
\end{equation}

For $p=\infty$, note that $(a,b) \times B(x,b^{1/2}) \subset (0,b) \times B(x,b^{1/2})$, so
\begin{align*}
    F(x) 
    &\lesssim b^{\frac{n}{4}} \max\{a^\beta,b^\beta\} \left ( \int_a^b \fint_{B(x,b^{1/2})} |t^{-\beta} u(t,y)|^2 dtdy \right )^{1/2} \\
    &\leq b^{\frac{n}{4}} \max\{a^\beta,b^\beta\} \sup_{B:x \in B} \left ( \int_0^{r(B)^2} \fint_B |t^{-\beta} u(t,y)|^2 dtdy \right )^{1/2}.
\end{align*}
Taking $L^\infty$-norm on both sides implies \eqref{e:L2loc_tent_precise}.
\end{proof}

The next lemma is a standard result for such parabolic equations, and we use the form given in  \cite[Proposition 3.6]{auscher-monniaux-portal2019existence}.

\begin{lemma}
    \label{lemma:Caccioppoli}
    Let $0<a<b<\infty$, $B \subset \bR^n$ be a ball, and $f$ be in $L^2((a,b) \times 2B)$. Let $u \in L^2((a,b);H^1(2B))$ be a weak solution to \eqref{e:def_ic} with source term $f$ in $(a,b) \times 2B$. Then, $u$ lies in $C([a,b];L^2(B))$ and
    \begin{align*}
        \|u(b)\|_{L^2(B)}^2 \lesssim
        &\left ( \frac{\Lambda_1^2}{\Lambda_0 r(B)^2}+\frac{1}{b-a} \right ) \int_a^b \|u(s)\|_{L^2(2B)}^2 ds \\
        &+ (b-a) \int_a^b \|f(s)\|_{L^2(2B)}^2 ds.
    \end{align*}
\end{lemma}

We also recall a standard corollary of Banach--Steinhaus theorem.

\begin{lemma}
    \label{lemma:pairing_cv}
    Let $X$ be a Banach space and $X^\ast$ be its dual. Let $(x_j)$ be a sequence in $X$ that converges to $x$ in $X$ and $(y_j)$ be a sequence in $X^\ast$ that weakly$^\ast$ converges to $y$ in $X^\ast$. Then, the sequence of pairing $(\langle y_j,x_j \rangle)$ converges to $\langle y,x \rangle$.
\end{lemma}

We now prove Theorem \ref{thm:unique_ic}.

\begin{proof}[Proof of Theorem \ref{thm:unique_ic}]
    Let $u$ and $\tilu$ be two global weak solutions in $T^{p}_{\beta+1}$ to \eqref{e:def_ic} with same source term $f \in T^p_\beta$. Define $v:=u-\tilu$. It lies in $L^2_{\loc}((0,\infty);H^1_{\loc}(\bR^n)) \cap T^{p}_{\beta+1}$ and is a global weak solution to \eqref{e:def_ic} with null source term. For $0<a<b<\infty$ and $\gamma>0$, we have
    \[ \int_{\bR^n} \left ( \int_a^b \int_{B(x,b^{1/2})} |v(t,y)|^2 dtdy \right )^{1/2} e^{-\gamma |x|^2} dx \lesssim_{\gamma,p,a,b,\beta} \|v\|_{T^{p}_\beta} < \infty \]
    by Lemma \ref{lemma:embed_tent_L2beta(a,b)_p<2} for $p_L(\beta)<p \leq 2$ and Lemma \ref{lemma:homotopy-condition} for $2 \leq p \leq \infty$. Thus, \cite[Theorem 5.1]{auscher-monniaux-portal2019existence} ensures that the homotopy identity holds. More precisely, for $0<s<t<\infty$ and $h \in C_\rmc^\infty(\bR^n)$, we have
    \begin{equation}
        \int_{\bR^n} v(t,x) \ovh(x) dx = \int_{\bR^n} v(s,x) \overline{((e^{-(t-s)L})^\ast h)}(x) dx.
    \label{e:ic_homotopy}
    \end{equation}

    We now take limits as $s \to 0$ on both sides to show $v(t)=0$ in $\scrD'(\bR^n)$, hence proving $v \equiv 0$. The proof is divided into three cases. \\
    
    \paragraph{Case 1: $2 \leq p < \infty$}
    Let $\delta>0$ be a constant. Using \eqref{e:Epdelta_change} for $t \in (0,\delta]$, we have $\|v(t)\|_{E^p_\delta} \lesssim \|v(t)\|_{E^p_{t/16}}$. Since $v$ lies in $L^2_{\loc}((0,\infty);H^1_{\loc}(\bR^n))$, Lemma \ref{lemma:Caccioppoli} yields for any $t \in (0,\delta]$,
    \begin{equation}
        \begin{aligned}
            &\left ( \int_{\bR^n} \left ( \fint_{B(x,\frac{\sqrt{t}}{4})} |v(t,y)|^2 dy \right )^{p/2} dx \right )^{1/p} \\ 
            &\quad\lesssim \left ( \int_{\bR^n} \left ( \fint_{t/2}^t \fint_{B(x,\frac{\sqrt{t}}{2})} |v(s,y)|^2 dsdy \right )^{p/2} dx \right )^{1/p} 
            \\ 
            &\quad\lesssim t^{\beta+\frac{1}{2}} \|v\|_{T^{p}_{\beta+1}}.
        \end{aligned}
        \label{e:unique_v(t)_Ep_p>2}
    \end{equation}
    Hence $\|v(t)\|_{E^p_{t/16}}\lesssim t^{\beta+\frac{1}{2}} \|v\|_{T^{p}_{\beta+1}}.$
        Thus, $v(t)$ tends to 0 in $E^p_\delta$ as $t \to 0$. Thanks to \cite[Lemma 4.7]{auscher-monniaux-portal2019existence}, we have $(e^{-(t-s)L})^\ast h$ tends to $(e^{-tL})^\ast h$ in $E^{p'}_\delta$ as $s \to 0$. By Lemma \ref{lemma:pairing_cv} with $X=E^p_\delta$, the right-hand side of \eqref{e:ic_homotopy} converges to 0 as $s \to 0$, and we obtain $v(t)=0$ in $\scrD'(\bR^n)$. \\

    \paragraph{Case 2: $p=\infty$} Let $\delta>0$ be a constant. For any $t \in (0,\delta]$,
    \begin{align*}
    \|v(t)\|_{E^\infty_\delta}
    &\lesssim \|v(t)\|_{E^\infty_{t/16}} = \esssup{x \in \bR^n} \left ( \fint_{B(x,\frac{\sqrt{t}}{4})} |v(t,y)|^2 dy \right )^{1/2} \\ 
    &\lesssim \esssup{x \in \bR^n} \left ( \fint_{t/2}^t \fint_{B(x,\frac{\sqrt{t}}{2})} |v(s,y)|^2 dsdy \right )^{1/2} \lesssim t^{\beta+\frac{1}{2}} \|v\|_{T^{\infty}_{\beta+1}}.
    \end{align*}
    Thus, $v(t)$ converges to 0 in $E^\infty_\delta$ as $t \to 0$. It is also shown in \cite[Lemma 4.7]{auscher-monniaux-portal2019existence} that $(e^{-(t-s)L})^\ast h$ tends to $(e^{-tL})^\ast h$ in $E^1_\delta$ as $s \to 0$. Thus, we get $v(t)=0$ in $\scrD'(\bR^n)$ by the same argument as in Case 1. \\

    \paragraph{Case 3: $p_L(\beta)<p<2$} 
    We claim that,
    \begin{equation}
        \|v(t)\|_p \lesssim t^{\beta+\frac{1}{2}} \|v\|_{T^{p}_{\beta+1}}, \quad \forall 0<p \leq 2, ~ \forall t>0.
        \label{e:v(t)Lp}
    \end{equation}
    Indeed, using H\"older's inequality and \eqref{e:unique_v(t)_Ep_p>2}, we have
    \[ \|v(t)\|_p = \left ( \int_{\bR^n} \fint_{B(x,\frac{\sqrt{t}}{4})} |v(t,y)|^p dydx \right )^{1/p} \lesssim t^{\beta+\frac{1}{2}} \|v\|_{T^{p}_{\beta+1}}. \]
    The claim hence follows. Moreover, Lemma \ref{lemma:Caccioppoli} and Lemma \ref{lemma:embed_tent_L2beta(a,b)_p<2} imply
    \begin{align*}
    \|v(t)\|_{L^2(B(x,R))}^2
     &\lesssim \left ( \frac{\Lambda_1^2}{\Lambda_0 R^2}+\frac{1}{t} \right ) t^{2(\beta+1)} \int_{t/2}^t \|s^{-(\beta+1)}v(s)\|_{L^2(B(x,2R))}^2 ds \\ 
     &\lesssim \left ( \frac{\Lambda_1^2}{\Lambda_0 R^2}+\frac{1}{t} \right ) t^{2(\beta+1)+n[2,p]} \|v\|_{T^{p}_{\beta+1}}^2.
    \end{align*}
    By letting $R \to \infty$, we get
    \begin{equation}
        \|v(t)\|_2 \lesssim t^{\beta+\frac{1}{2}-\frac{n}{2}[p,2]} \|v\|_{T^{p}_{\beta+1}}, \quad \forall 0<p \leq 2, ~ \forall t>0.
        \label{e:v(t)L2}
    \end{equation}
    As $p_L(\beta)<p<2$ implies
    \[ \frac{1}{p}-\frac{2\beta+1}{n} < \frac{1}{p_-(L)}, \]
    there exists $q \in (p,2) \cap (p_-(L),2)$ such that
    \[ \frac{1}{p} - \frac{2\beta+1}{n} < \frac{1}{q}. \]
    Interpolating \eqref{e:v(t)Lp} and \eqref{e:v(t)L2} yields
    \[ \|v(t)\|_q \lesssim t^{\beta+\frac{1}{2}-\frac{n}{2}[p,q]} \|v\|_{T^{p}_{\beta+1}}, \quad \forall q \in (p,2), ~ \forall t>0. \]
    Thus, $v(s)$ tends to 0 in $L^q(\bR^n)$ as $s \to 0$. Thanks to \cite[Proposition 3.15]{Auscher2007Memoire}, we also know that $(e^{-(t-s)L})^\ast h$ tends to $(e^{-tL})^\ast h$ in $L^{q'}(\bR^n)$ as $s \to 0$ since $2<q'<p_+(L^\ast) = (p_-(L))'$. Using \eqref{e:ic_homotopy} and Lemma \ref{lemma:pairing_cv} with $X=L^q(\bR^n)$, we conclude that $v(t)=0$ in $\scrD'(\bR^n)$. 
    
    This completes the proof.
\end{proof}
\section{Maximal regularity on tent spaces}
\label{sec:mr}

In this section, we prove maximal regularity of the system \eqref{e:def_ic}, \textit{i.e.}, Theorem \ref{thm:ic_sol} \eqref{item:ic_max_reg}, and apply it to investigate the boundary behavior of global weak solutions to \eqref{e:def_ic}, see \eqref{item:ic_boundary_whitney} and \eqref{item:ic_boundary_LpEp}. Throughout the section, the notation is the same as in Section \ref{sec:inhomo} without special mention.

\subsection{Maximal regularity operator}
\label{ssec:MRO_SIO}
Define the \textit{maximal regularity operator} $\cL_0$ on $L^2((0,\infty);D(L))$ by
\begin{equation}
    \cL_0(f)(t,y) := \int_0^t (Le^{-(t-s)L}f(s))(y) ds, \quad \mathrm{a.e.}\ (t,y)\in \bR^{1+n}_+.
    \label{e:def_mr}
\end{equation}
The integral {makes sense as an $L^2(\bR^n)$-valued Bochner integral.} Recall that the celebrated de Simon's theorem \cite{deSimon1964MRO_L2} asserts $\cL_0$ can be extended by density to a bounded operator on $L^2(\bR^{1+n}_+)$. Denote by $\Tilde{\cL}_0$ this extension.

\begin{lemma}
    \label{lemma:L0_sio}
    $\Tilde{\cL}_0$ belongs to $\sio^{0+}_{2,q,\infty}$ for any $q \in (p_-(L),p_+(L))$.
\end{lemma}

\begin{proof}
    The boundedness of $\tilde{\cL}_0$ on $L^2(\bR^{1+n}_+)$ is clear by construction, and  the function $(t,s) \mapsto \I_{\{t>s\}}(t,s) Le^{-(t-s)L}$ belongs to $\sk^0_{2,q,\infty}$ for any $q \in (p_-(L),p_+(L))$ by \cite[Proposition 3.15]{Auscher2007Memoire}. Thus, Lemma \ref{lem:sk_of_si_kappa<0} yields for any $f \in L^2_\rmb(\bR^{1+n}_+)$,
    \[ (t,x) \mapsto \int_0^t (Le^{-(t-s)L} f)(x) ds \]
    defines a function in $L^2_{\loc}(\overline{\bR^{1+n}_+} \setminus \pi(f))$. It only remains to prove the representation \eqref{e:def_T-representation}, or equivalently, for any $f \in L^2_\rmb(\bR^{1+n}_+)$ and $g \in C_\rmc^\infty({\bR^{1+n}_+})$ with $\pi(g) \cap \pi(f) = \varnothing$,
    \begin{equation}
        \left \langle \Tilde{\cL}_0(f),g \right \rangle_{L^2(\bR^{1+n}_+)} = \int_0^\infty \int_{\bR^n} \left( \int_0^t (L e^{-(t-s)L} f(s))(x) ds \right) \ovg(t,x) dtdx.
        \label{e:L0_representation_pair}
    \end{equation}
    
    To prove \eqref{e:L0_representation_pair}, we proceed by putting the following two observations together. Fix $g \in C_\rmc^\infty({\bR^{1+n}_+})$. 
    
    The first observation is that for any $f \in L^2((0,\infty);D(L))$,
    \begin{equation}
        \left \langle \cL_0(f),g \right \rangle_{L^2(\bR^{1+n}_+)}=\left \langle A\cL_{1/2}(f), \nabla g \right \rangle_{L^2(\bR^{1+n}_+)}.
        \label{e:L0(f),g_D(L)case}
    \end{equation}
    Indeed, since $\cL_{1/2}(f)$ is given by an $L^2(\bR^n)$-valued Bochner integral, all the following integrals converge absolutely, so Fubini's theorem ensures
    \begin{align*}
        \left \langle \cL_0(f),g \right \rangle_{L^2(\bR^{1+n}_+)} 
        &= \int_0^\infty  \int_0^t \left \langle Le^{-(t-s)L}f(s),g(t) \right \rangle_{L^2(\bR^n)} dsdt \\
        &= \int_0^\infty  \int_0^t \left \langle A\nabla e^{-(t-s)L}f(s), \nabla g(t) \right \rangle_{L^2(\bR^n)} dsdt\\
        &=\left \langle A\cL_{1/2}(f), \nabla g \right \rangle_{L^2(\bR^{1+n}_+)}.
    \end{align*}

   The second observation is that when $f$ lies in $L^2_\rmb(\bR^{1+n}_+)$ with $\pi(f)$  disjoint with $\pi(g)$, we have the identity
    \begin{equation}
        \begin{aligned}
            &\left \langle A\cL_{1/2}(f), \nabla g \right \rangle_{L^2(\bR^{1+n}_+)} \\
            &\quad = \int_0^\infty \int_0^t \int_{\bR^n} (A\nabla e^{-(t-s)L} f(s))(x) \cdot \overline{\nabla g}(t,x) dxdsdt \\
            &\quad =- \int_0^\infty \int_0^t \int_{\bR^n} \Div(A\nabla e^{-(t-s)L} f(s))(x) \overline{g}(t,x) dxdsdt =: I
        \end{aligned}
        \label{e:Anabla->div(Anabla)}
    \end{equation}
    Indeed, all the integrals converge absolutely as for fixed $(t,s)$, both $x \mapsto (A\nabla e^{-(t-s)L} f(s))(x)$ and $x \mapsto \Div (A\nabla e^{-(t-s)L} f(s))(x)$ are integrable on the support of $ g(t)$, using Lemma \ref{lem:ext_L2_kappa>0} with $\kappa=1/2$ and Lemma \ref{lem:sk_of_si_kappa<0} with $\kappa=0$. Thus, the second equality follows by integration by parts for the divergence.
    
    We now prove \eqref{e:L0_representation_pair} as follows. By a limiting argument from the first observation, we have that for any $f\in L^2(\bR^{1+n}_+)$, \eqref{e:L0(f),g_D(L)case} becomes
    \begin{equation}
        \left \langle \Tilde{\cL}_0(f),g \right \rangle_{L^2(\bR^{1+n}_+)} = \left \langle A\cL_{1/2}(f), \nabla g \right \rangle_{L^2(\bR^{1+n}_+)}
        \label{e:L0->AL1/2_L2}
    \end{equation}
    Indeed, the argument is similar to Lemma \ref{lemma:L1/2=nabla_L1}, but using the continuity of $\cL_{1/2}$ and $\tilde{\cL}_0$. We hence leave details to the reader.

    Next, for any $f \in L^2_\rmb(\bR^{1+n}_+)$ with $\pi(f)$ disjoint with $\pi(g)$, \eqref{e:L0->AL1/2_L2} and \eqref{e:Anabla->div(Anabla)} yield
    \[ \left \langle \Tilde{\cL}_0(f),g \right \rangle_{L^2(\bR^{1+n}_+)} = I. \]
    So it only remains to observe that for a.e. $s<t$ and $x \in \bR^n$,  
    \[ \Div(A\nabla e^{-(t-s)L} f(s))(x)= (Le^{-(t-s)L} f(s))(x) \]
    by definition of $L$, as $f(s)$ belongs to $L^2(\bR^n)$.
\end{proof}

We then consider a further extension of $\tilde{\cL}_0$ to tent spaces.
\begin{prop}
    \label{prop:ext_L0_tent}
    Let $\beta>-1/2$ and $p_L(\beta) < p \leq \infty$ . Then, $\tilde{\cL}_0$ can be extended to a bounded operator on $T^{p}_\beta$.
\end{prop}

\begin{proof}
    Lemma \ref{lemma:L0_sio} says $\tilde{\cL}_0$ lies in $\sio^{0+}_{2,q,\infty}$ when $p_-(L)<q<p_+(L)$, and \cite[Theorem 1.3]{Auscher-Axelsson2011_MR_L2beta} shows $\tilde{\cL}_0$ is bounded on $L^2_\beta(\bR^{1+n}_+)$ for any $\beta>-1/2$. We hence conclude by invoking Proposition \ref{prop:sio_tent_p<2} and Corollary \ref{cor:sio_tent_p>2}.
\end{proof}

The extension of $\tilde{\cL}_0$ to $T^p_\beta$ is still denoted by $\tilde{\cL}_0$ in the sequel. Write $\cL_1$ and $\cL_{1/2}$ for their extensions to $T^p_\beta$. Let us prove Theorem \ref{thm:ic_sol} \eqref{item:ic_max_reg}.

\begin{proof}[Proof of Theorem \ref{thm:ic_sol} \eqref{item:ic_max_reg}]
    Let $f$ be in $T^p_\beta$ and $u=\cL_1(f)$. It suffices to identify $-\Div(A \nabla u)$ with $\tilde\cL_0(f)$ in $\scrD'(\bR^{1+n}_+)$. Indeed, if it holds, then Proposition \ref{prop:ext_L0_tent} yields
    \[ \|\Div(A \nabla u)\|_{T^p_\beta} = \|\Tilde{\cL}_0(f)\|_{T^p_\beta} \lesssim \|f\|_{T^p_\beta}. \]
    Moreover, \eqref{item:ic_sol} implies $\partial_t u = \Div(A \nabla u) + f$ in $\scrD'(\bR^{1+n}_+)$. Thus, it follows directly that $\partial_t u$ lies in $T^p_\beta$ with
    \[ \|\partial_t u\|_{T^p_\beta} \lesssim \|\Div(A\nabla u)\|_{T^p_\beta} + \|f\|_{T^p_\beta} \lesssim \|f\|_{T^p_\beta}. \]
    
    To prove the identification, we first observe that
    \begin{equation}
        \tilde{\cL}_0(f) = -\Div(A\cL_{1/2}(f)) \quad \text{ in } \scrD'(\bR^{1+n}_+).
        \label{e:id_tilL0=div(AL1/2)_tent}
    \end{equation}
    In fact, it follows from \eqref{e:L0->AL1/2_L2} with a limiting argument applied to a sequence in $T^p_\beta \cap L^2(\bR^{1+n}_+)$ approximating $f$ ( ${\weakstar}$ if $p=\infty$, see Case 2 of Lemma \ref{lemma:ext_L1-1/2_tent}). The second observation is that
    \begin{equation}
        -\Div(A\cL_{1/2}(f)) = -\Div(A\nabla \cL_1(f)) \quad \text{ in } \scrD'(\bR^{1+n}_+).
        \label{e:id_divAL1/2=Lu_tent}
    \end{equation}
    Indeed, Lemma \ref{lemma:ext_L1-1/2_tent} says $\cL_{1/2}(f)=\nabla \cL_1(f)$ in $T^p_{\beta+\frac{1}{2}}$, hence in $L^2_{\loc}(\bR^{1+n}_+)$, so \eqref{e:id_divAL1/2=Lu_tent} follows by multiplying by $A$ and applying the divergence. Combining \eqref{e:id_tilL0=div(AL1/2)_tent} and \eqref{e:id_divAL1/2=Lu_tent} yields $\tilde{\cL}_0(f)=-\Div(A\nabla u)$ as desired.
\end{proof}

\subsection{Boundary behavior}
\label{ssec:boundary_ic}
In this section, we prove Theorem \ref{thm:ic_sol} \eqref{item:ic_boundary_whitney} and \eqref{item:ic_boundary_LpEp}. The first result is valid for arbitrary function in $T^p_{\beta+1}$.

\begin{lemma}
    \label{lemma:whitney-boundary_tent}
    Let $0<p \leq \infty$ and $\beta>-1/2$. Let $u$ be in $T^{p}_{\beta+1}$. Then, for a.e. $x \in \bR^n$,
    \[ \lim_{t \rightarrow 0} \left( \fint_{W(t,x)} |u(s,y)|^2 dsdy \right)^{1/2} = 0. \]
\end{lemma}

\begin{proof}
    First consider $0<p<\infty$. For any $t>0$ and $x \in \bR^n$, we have
    \[ \fint_{W(t,x)} |u(s,y)|^2 dsdy \lesssim \fint_t^{2t} ds \fint_{B(x,s^{1/2})} |u(s,y)|^2 dy \lesssim t^{2\beta+1} \cA_{\beta+1;2}(u)(x)^{2}, \]
    where $\cA_{\beta+1;2}$ is defined in \eqref{e:def_A_concial_square}. Since $\|\cA_{\beta+1;2}(u)\|_p \eqsim \|u\|_{T^{p}_{\beta+1}} < \infty$, $\cA_{\beta+1}(u)(x)$ is finite for a.e. $x \in \bR^n$, so the right-hand side tends to 0 when $t \to 0$ as $\beta>-1/2$.
    
    Then for $p=\infty$, note that $W(t,x) \subset (0,2t) \times B(x,(2t)^{1/2})$, so
    \[ \fint_{W(t,x)} |u(s,y)|^2 dsdy \lesssim t^{2\beta+1} \int_t^{2t} ds \fint_{B(x,(2t)^{1/2})} |s^{-(\beta+1)} u(s,y)|^2 dy. \]
    Taking esssup on $x \in \bR^n$, we get
    \[ \esssup{x \in \bR^n} \fint_{W(t,x)} |u(s,y)|^2 dsdy \lesssim t^{2\beta+1} \|u\|_{T^{\infty}_{\beta+1}}^{2}. \]
    This completes the proof.
\end{proof}

\begin{lemma}
    \label{lemma:tent-embed-L2loc_bar}
    For any $p \in [2,\infty]$ and $\beta \geq -1$, $T^{p}_{\beta+1}$ is contained in $L^2_{\loc}(\overline{\bR^{1+n}_+})$. More precisely, for ball $B \subset \bR^n$ and $0<T<r(B)^2$,
    \begin{equation}
        \label{e:tent-L2(0,T)*B}
        \left ( \int_0^T \int_B |u(t,y)|^2 dtdy \right )^{1/2} \lesssim_p T^{\beta+1} |B|^{[2,p]} \|u\|_{T^{p}_{\beta+1}}.
    \end{equation}
\end{lemma}

\begin{proof}
    By definition, it is trivial if $p=\infty$. We thus assume $2 \leq p < \infty$. For any $\alpha>0$ and ball $B \subset \bR^n$, define the \textit{tent} on $B$ of \textit{aperture} $\alpha$ (and \textit{homogeneity} $m=2$) as
    \[ T_\alpha(B):=\{(t,y) \in \bR^{1+n}_+: 0<\alpha t^{1/2} \leq \dist(y,B^c)\}. \]
    Observe that $(0,T) \times B \subset T_\alpha(2B)$ with $\alpha:=T^{-1/2} r(B)>1$. Moreover, for any $(t,y) \in T_\alpha(2B)$, since $\alpha t^{1/2} < r(B) < \dist(y,(2B)^c)$, we know that $B(y,\alpha t^{1/2}) \subset 2B$. In other words,
    \[ \int_{2B} \I_{B(0,1)}\left ( \frac{x-y}{\alpha t^{1/2}} \right ) dx \eqsim (\alpha t^{1/2})^n. \]
    With the help of these two observations, we get
    \begin{align*}
        \|u\|_{L^2((0,T) \times B)}^2
        &\lesssim \alpha^{-n} \int_{T_\alpha(2B)} |u(t,y)|^2 dt\frac{dy}{t^{n/2}} \int_{2B} \I_{B(0,1)}\left ( \frac{x-y}{\alpha t^{1/2}} \right ) dx \\
        &\leq \alpha^{-n} \int_{2B} dx \int_0^\infty \int_{B(x,\alpha t^{1/2})} |u(t,y)|^2 dt\frac{dy}{t^{n/2}} \\
        &\le T^{2(\beta+1)} \alpha^{-n} \int_{2B} \cA^{(\alpha)}_{\beta+1;2}(u)(x)^2 dx \\
        &\lesssim_p T^{2(\beta+1)} \alpha^{-n} |B|^{2[2,p]} \|\cA^{(\alpha)}_{\beta+1;2}(u)\|_p^2 \lesssim T^{2(\beta+1)} |B|^{2[2,p]} \|u\|_{{T^{p}_{\beta+1}}}^2.
    \end{align*}
    The last inequality follows from the change of aperture, cf. \eqref{e:tent_change_ap}.
\end{proof}

\begin{cor}
    \label{cor:Cesaro_L2(B)_tent}
    Let $2 \leq p \leq \infty$ and $\beta \geq -1$. Let $B \subset \bR^n$ be a ball and $0<T<r(B)^2$. For any $u \in T^{p}_{\beta+1}$,
    \begin{equation}
        \fint_0^T \|u(t)\|_{L^2(B)} dt \lesssim_\beta T^{\beta+\frac{1}{2}} |B|^{[2,p]} \|u\|_{T^{p}_{\beta+1}}.
        \label{e:Cesaro_L2(B)_tent}
    \end{equation}
    In particular, if $\beta>-1/2$, $u(t)$ tends to 0 in a Ces\`aro mean in $L^2_{\loc}(\bR^n)$ as $t \to 0$.
\end{cor}

\begin{proof}
    It suffices to prove \eqref{e:Cesaro_L2(B)_tent}, which is a direct consequence of Lemma \ref{lemma:tent-embed-L2loc_bar} by Jensen's inequality as
    \begin{align*}
        \fint_0^T \|u(t)\|_{L^2(B)} dt 
        \leq \left ( \fint_0^T \| u(t)\|_{L^2(B)}^2 dt \right )^{1/2} 
        \lesssim T^{\beta+\frac{1}{2}} |B|^{[2,p]} \|u\|_{T^{p}_{\beta+1}}.
    \end{align*}
    This completes the proof.
\end{proof}

\begin{cor}
    \label{cor:Einfty_est_sol}
    Let $2 \leq p \leq \infty$ and $\beta \geq -1$. Let $u$ be in $T^{p}_{\beta+1}$ with $\partial_t u \in T^{p}_\beta$. Then, for any $\delta>0$ and $t \in (0,\delta)$,
    \begin{equation}
        \|u(t)\|_{E^\infty_\delta} \lesssim \delta^{-\frac{n}{2p}} t^{\beta+\frac{1}{2}} \big ( \|u\|_{T^p_{\beta+1}} + \|\partial_t u\|_{T^p_\beta} \big ).
        \label{e:Einfty_est_sol}
    \end{equation}
\end{cor}

\begin{proof}
    Fix a ball $B \subset \bR^n$. For any $T>0$, by \eqref{e:tent_truncate_Lq}, both of $u$ and $\partial_t u$ belong to $L^2((T/8,2T) \times B)$. Thus, for any $T' \in [T/4,T/2]$, we have
    \[ u(T) = u(T') + \int_{T'}^T \partial_t u(t) dt \]
    valued in $L^2(B)$. Taking average with respect to $T'$, we get
    \[ u(T) = \fint_{T/4}^{T/2} u(T') dT' + \fint_{T/4}^{T/2} dT' \int_{T'}^T \partial_t u(t) dt. \]
    Since $T/4 \leq T' \leq \min\{t,T/2\}$, Fubini's theorem yields
    \[ \|u(T)\|_{L^2(B)} \lesssim \fint_{T/4}^{T/2} \|u(T')\|_{L^2(B)} dT' + \fint_{T/4}^T t\|\partial_t u(t)\|_{L^2(B)} dt. \]
    Note that both of $u$ and $t\partial_t u$ lie in $T^{p}_{\beta+1}$. Using Corollary \ref{cor:Cesaro_L2(B)_tent},  we deduce
    \[ \|u(t)\|_{E^\infty_\delta} \eqsim \delta^{-\frac{n}{2}} \sup_B \|u(t)\|_{L^2(B)} \lesssim \delta^{-\frac{n}{2p}} t^{\beta+{\frac{1}{2}}} \big( \|u\|_{T^p_{\beta+1}} + \|\partial_t u\|_{T^p_\beta} \big), \]
    where $B$ describes all balls in $\bR^n$ with $r(B)=\delta^{1/2}$.
\end{proof}

We may finish the proof of Theorem \ref{thm:ic_sol}.

\begin{proof}[Proof of Theorem \ref{thm:ic_sol} \eqref{item:ic_boundary_whitney} and \eqref{item:ic_boundary_LpEp}]
    Statement \eqref{item:ic_boundary_whitney} is clear by Lemma \ref{lemma:whitney-boundary_tent}. For \eqref{item:ic_boundary_LpEp}, first note that the same deduction as for  \eqref{e:unique_v(t)_Ep_p>2} implies
    \begin{equation}
        \|u(t)\|_{E^p_{t/16}} \lesssim t^{\beta+\frac{1}{2}} ( \|u\|_{T^p_{\beta+1}} + \|f\|_{T^p_\beta}).
        \label{e:bb_u(t)Ept/16}
    \end{equation}
    Then the discussion is divided into two cases. \\
    
    \paragraph{Case 1: $p_L(\beta)<p \leq 2$} 
    Observe that $E^p_{t/16}$ embeds into $L^p(\bR^n)$ by H\"older's inequality, so $u(t)$ tends to 0 in $L^p(\bR^n)$ as $t \to 0$. \\

    \paragraph{Case 2: $2<p \leq \infty$}
    Fix $\delta>0$ and $t \in (0,\delta)$. Applying \eqref{e:Epdelta_change} and \eqref{e:bb_u(t)Ept/16} yields 
    \[ \|u(t)\|_{E^p_\delta} \lesssim \|u(t)\|_{E^p_{t/16}} \lesssim t^{\beta+\frac{1}{2}} ( \|u\|_{T^p_{\beta+1}} + \|f\|_{T^p_\beta} ). \]
    Moreover, \eqref{e:Einfty_est_sol} also holds for $u$ by Corollary \ref{cor:Einfty_est_sol} as $u \in T^p_{\beta+1}$ (by \eqref{item:ic_regularity}) and $\partial_t u \in T^p_\beta$ (by \eqref{item:ic_max_reg}). Interpolation of slice spaces hence implies $u(t)$ tends to 0 as $t \to 0$ in $E^q_\delta$ for any $q \in [p,\infty]$. 
    
    This completes the proof.
\end{proof}
\section{Further remarks}
\label{sec:remarks}

In fact, our results can be generalized in several directions. Let us briefly mention two possible results.

(1) For any $\beta \in \bR$, $T\in (0,\infty]$, and $u \in L^2_{\loc}((0,T)\times \bR^{n})$, the \textit{(local) weighted Kenig--Pipher non-tangential maximal function} $\cN_{\beta,T}$ is defined by
\[ \cN_{\beta,T}(u)(x):=\sup_{0<t<T} \left( \fint_{t/2}^t \fint_{B(x,t^{1/2})} |s^{-\beta} u(s,y)|^2 dsdy \right)^{1/2}. \]
For any $p \in (0,\infty]$, the \textit{(local) weighted Kenig--Pipher space} $X^p_{\beta,T}$ consists of measurable functions $u \in L^2_{\loc}((0,T)\times \bR^{n})$ for which
\[ \|u\|_{X^p_{\beta,T}} := \|\cN_{\beta,T}(u)\|_p < \infty. \]
Simple calculation shows that for any $p \in (0,\infty]$, $\beta \in \bR$, and $T \in (0,\infty]$, $T^p_{\beta+\frac{1}{2}} \hookrightarrow X^p_{\beta,\infty} \hookrightarrow X^p_{\beta,T}$. Thus, when $\beta>-1/2$ and $p_L(\beta)<p \le \infty$, for any $f \in T^p_\beta$, the global weak solution $u$ given by Theorem \ref{thm:ic_sol} lies in $X^p_{\beta+\frac{1}{2},T}$. In fact, only knowing that $\I_{(0,T]}f\in T^p_\beta$ suffices to conclude that $u\in X^p_{\beta+\frac{1}{2},T}$.   Also, there is uniqueness in $X^p_{\beta+\frac{1}{2}, T}$.

\begin{prop}
    Let $\beta>-1/2$, $T\in (0,\infty]$, and $p_L(\beta)<p \le \infty$. Then any weak solution $u \in X^p_{\beta+\frac{1}{2}, T}$ to $\partial_t u - \Div (A\nabla u)=0$ vanishes.
\end{prop}

It suffices to follow our argument of Theorem \ref{thm:unique_ic},  with little modifications. The only noticeable one is that we use (twice) Lemma \ref{lemma:embed_tent_L2beta(a,b)_p<2} on strips $(\frac{t}{2},t) \times \bR^n$ for $0<t<T$ with the remark that 
\[ \|\I_{(\frac{t}{2},t)} u\|_{T^p_{\beta+1}} \eqsim \|\I_{(\frac{t}{2},t)} u\|_{X^p_{\beta+\frac{1}{2},T}}. \]

(2) Observe that Corollary \ref{cor:sio_tent_p>2} and \eqref{e:L1-1/2_sio} imply $\cL_1$ extends to a bounded operator from $T^{\infty}_{\beta,([p,1])}$ to $T^{\infty}_{\beta+1,([p,1])}$ for any $0<p \le 1$. In fact, we can also establish well-posedness within this range using the same method as what we did for $f \in T^\infty_\beta$.
\begin{prop}
    Let $\beta>-1/2$ and $0<p \le 1$. Then for any $f \in T^\infty_{\beta;([p,1])}$, there exists a unique global weak solution $u \in T^\infty_{\beta+1;([p,1])}$ to \eqref{e:def_ic} with source term $f$.
\end{prop}

Similarly, all the properties of Theorem \ref{thm:ic_sol} can be accordingly established. We only mention that \eqref{item:ic_boundary_LpEp} turns out to be that:
\begin{enumerate}
    \item[(e')] $u(t)$ converges to 0 as $t \to 0$ in $E^\infty_\delta$ for any $\delta>0$.
\end{enumerate}

\bibliographystyle{alpha}
\bibliography{sample}

\end{document}